\documentclass[psamsfonts]{amsart}

\usepackage{amssymb,amsfonts}
\usepackage[all,arc]{xy}
\usepackage{enumerate}
\usepackage{mathrsfs}
\usepackage{amsmath}
\usepackage{fullpage}
\usepackage{hyperref}
\usepackage{graphicx}

\newtheorem{theorem}{Theorem}

\newtheorem{claim}{Claim}

\newtheorem{definition}{Definition}
\newtheorem{example}{Example}

\newtheorem{lemma}{Lemma}

\newtheorem{proposition}{Proposition}
\newtheorem{remark}{Remark}

\numberwithin{equation}{section}

\title{Singularities of plane complex curves and limits of K\"ahler metrics with cone singularities. I: Tangent Cones}

\author{Martin de Borbon}

\date{OCTOBER, 2016}

\begin{document}

\begin{abstract}

The goal of this article is to provide a construction and classification, in the case of two complex dimensions, of the possible tangent cones at points of limit spaces of non-collapsed sequences of K\"ahler-Einstein metrics with cone singularities. The proofs and constructions are completely elementary, nevertheless they have an intrinsic beauty. In a few words; tangent cones  correspond to spherical metrics with cone singularities in the projective line by means of the K\"ahler quotient construction with respect to the $S^1$-action generated by the Reeb vector field, except in the irregular case $\mathbb{C}_{\beta_1} \times \mathbb{C}_{\beta_2}$ with \( \beta_2 / \beta_1 \notin \mathbb{Q} \). 

\end{abstract}

\maketitle


\section{Introduction}
K\"ahler-Einstein (KE) metrics, and more generally constant scalar curvature and extremal K\"ahler metrics, are canonical metrics on polarized projective varieties and serve as a bridge between differential and algebraic geometry. More recently, after fundamental work of Donaldson \cite{donaldson1}, much of the theory has been extended to the setting of KE metrics with cone singularities along a divisor -which were previously introduced by Tian in \cite{Tian}-. A remarkable application is the proof of existence of KE metrics on K-stable Fano manifolds, through the deformation of the cone angle method (see \cite{CDS0}); but besides that, KE metrics with cone singularities (KEcs) have intrinsic interest -as canonical metrics on pairs of  projective varieties together with  divisors-. 

A major achievement in K\"ahler geometry in the past few years is the proof of a conjecture of Tian -on uniform lower bounds on Bergman kernels- which endows Gromov-Hausdorff limits, of non-collapsed sequences of smooth KE metrics on projective varieties, with an induced \emph{algebraic} structure -see \cite{DS}-. 
There is a strong interaction between the non-collapsed metric degenerations on the differential geometric side and the so-called log terminal singularities on the algebraic counterpart. The situation is better understood in two complex dimensions; Odaka-Spotti-Sun \cite{OSS} have shown that  the Gromov-Hausdorff compactifications of KE metrics on Del Pezzo surfaces agree with algebraic ones. One would expect then parallel results for  KEcs. The new feature is that the curves along which the metrics have cone singularities might now degenerate;  and we want to relate the metric degeneration with the theory of singularities of plane complex curves. This paper is a first step along this road and we concentrate in the study of tangent cones at points of limit spaces. 

Our main results are Propositions \ref{Hfm}, \ref{Sfm} and \ref{Clas} that follow. Proposition \ref{Sfm} follows immediately from \ref{Hfm}; while \ref{Clas} has already been established in \cite{panov} and its proof is included here only for the sake of completeness. The main interest is therefore in \ref{Hfm}.

We work on $\mathbb{C}^2$ with standard complex coordinates $z, w$. Let $d \geq 2$ and take $L_j = \{ l_j (z, w) =0 \}$ for $j=1, \ldots, d$ to be $d$ distinct complex lines through the origin  with defining linear equations $l_j$. Let $\beta_1, \ldots, \beta_d \in (0,1)$ satisfy the Troyanov condition
\begin{equation} \label{Tcondition}
0<2-d+\sum_{j=1}^d \beta_j < 2\min_{i} \beta_i 
\end{equation}
if $d \geq 3$ and $\beta_1=\beta_2$ if $d=2$.

\begin{proposition} \label{Hfm}
 There is a unique K\"ahler cone metric $g_F$ on $\mathbb{C}^2$ with apex at $0$ such that
 \begin{enumerate}
 	\item \label{item1} Its Reeb vector field generates the circle action $e^{it}(z,w)=(e^{ \frac{i t}{c}}z, e^{ \frac{it}{c}}w)$ for some constant $c>0$.
 	\item It has cone angle $2\pi \beta_j$ along $L_j$ for $j=1, \ldots, d$.
 	\item \label{item3} Its volume form is 
 	\begin{equation*} 
 	\mbox{Vol}(g_F) = |l_1|^{2\beta_1 -2} \ldots |l_d|^{2\beta_d -2} \frac{dzdw \overline{dzdw}}{4}.
 	\end{equation*}	
 \end{enumerate}
\end{proposition}

Item \ref{item3} implies that the metric $g_F$ is Ricci-flat and, since it is a Riemannian cone of real dimension four, it must be flat. Item \ref{item1} on the Reeb vector field implies that the maps $m_{\lambda}(z,w)=(\lambda z, \lambda w)$ for $\lambda >0$ must act by scalings of the metric, so that 

\begin{equation*} 
m_{\lambda}^{*} g_F = \lambda^{2c} g_F .
\end{equation*}
Condition \ref{item3} on the volume form implies that 
\begin{equation} \label{number c}
c=1-\frac{d}{2} +\sum_{j=1}^d \frac{\beta_j}{2};
\end{equation}
note that $0<c<1$.

 We move on to a slightly different situation. Take co-prime integers  $1 \leq p < q$. Let $d \geq 2$  and  $C_j = \{z^q =a_j  w^p \}$, $a_j \in \mathbb{C}$ for $j=1, \ldots, d-2$, be distinct complex curves; let $\beta_1, \ldots, \beta_{d-2} \in (0,1)$ and $\beta_{d-1}, \beta_d \in (0,1]$ be such that $\beta_1, \ldots, \beta_{d-2}, (1/q)\beta_{d-1}, (1/p)\beta_d$ satisfy the Troyanov condition \ref{Tcondition} if $ d\geq 3$ and $\beta_{d-1} / q = \beta_d /p$ if $d=2$. 

\begin{proposition} \label{Sfm}
There is a unique K\"ahler cone metric $\tilde{g}_F$ on $\mathbb{C}^2$ with apex at $0$ such that
	 \begin{enumerate}
	 	\item Its Reeb vector field generates the circle action $e^{it}(z,w)=(e^{ \frac{ipt}{\tilde{c}}}z, e^{ \frac{iqt}{\tilde{c}}}w)$ for some constant $\tilde{c}>0$.
	 	\item It has cone angle $2\pi \beta_j$ along $C_j$ for $j=1, \ldots, d-2$, $2\pi \beta_{d-1}$ along $\{z=0\}$ and $2 \pi \beta_d$ along $\{w=0\}$.
	 	\item Its volume form is 
	 	\begin{equation*} 
	 	\mbox{Vol}(\tilde{g}_F) = |z^q - a_1 w^p|^{2\beta_1 -2} \ldots |z^q - a_{d-2}w^p|^{2\beta_{d-2} -2} |z|^{2\beta_{d-1} -2} |w|^{2\beta_d -2} \frac{dzdw \overline{dzdw}}{4}.
	 	\end{equation*}
	 \end{enumerate}	
\end{proposition}
Proposition \ref{Sfm} follows from \ref{Hfm}, after pulling back by the map \( (z, w) \to (z^q, w^p) \).
Similar comments as those after Proposition \ref{Hfm} apply. For $\lambda >0$, let $\tilde{m}_{\lambda}(z,w)=(\lambda^p z, \lambda^q w)$. Then $\tilde{m}_{\lambda}^{*} \tilde{g}_F = \lambda^{2\tilde{c}} \tilde{g}_F$ with $\tilde{c} = pq \left(  1-d/2 +\sum_{j=1}^{d-2} \beta_j /2 + (1/2q)\beta_{d-1} + (1/2p)\beta_d \right) $. It is straightforward to include the case of curves like \( C = \{ z^n = w^m \} \)  with $m$ and $n$ not necessarily co-prime; simply let $m= dp$ and $n=dq$ with $p$ and $q$ co-prime, so that \( C= \cup_{j=1}^d \{ z^q = e^{2\pi i j / d} w^p \} \).
	
The last result asserts that Propositions \ref{Hfm} and \ref{Sfm} provide a complete list, up to finite coverings, of the Ricci-flat K\"ahler cone metrics with cone singularities (RFKCcs) in two complex dimensions, except for one case.

\begin{proposition} \label{Clas}
	Let $g_C = dr^2 + r^2 g_S$ be a RFKCcs and assume that its link is diffeomorphic to the 3-sphere; then there is an (essentially unique) holomorphic isometry of $((0, \infty) \times S^3, I, g_C)$ with one of the following 
	
	\begin{enumerate}
		
		\item  \emph{Regular case.} A metric $g_F$ given by Proposition \ref{Hfm}. 
		
		\item  \emph{Quasi-regular case.} A metric $\tilde{g}_F$ given by Proposition \ref{Sfm}. 
		
		\item  \emph{Irregular case.} $\mathbb{C}_{\beta_1} \times \mathbb{C}_{\beta_2}$ for some $0 < \beta_1 <1$ and $0 < \beta_2 \leq 1 $ with $\beta_1 / \beta_2 \notin \mathbb{Q}$.
		
	\end{enumerate}
	
\end{proposition}

The $g_S$ are spherical metrics on the 3-sphere with cone singularities along Hopf circles and $(p, q)$ torus knots. In Section \ref{SECTS}  we construct the $g_S$ as lifts of spherical metrics with cone singularities on the projective line by means of the Hopf map in the regular case and a Seifert map in the quasi-regular case. Propositions \ref{Hfm}, \ref{Sfm} and \ref{Clas} are then proved in Section \ref{SECTC}. 
Finally, in Section \ref{CURVES}, we discuss relations to the theory of singularities of plane complex curves and algebraic geometry.

After writing a first version of this paper, the author red Panov's article Polyhedral K\"ahler Manifolds \cite{panov}. Our results overlap substantially with the content of Section 3 in \cite{panov} and we refer to this article for a beautiful geometric exposition. Nevertheless our approach to Proposition \ref{Hfm} is slightly different from Panov's;  our proof goes along the lines of the well-known Calabi ansatz and suggests a higher dimensional generalization, replacing the spherical metrics with K\"ahler-Einstein metrics of positive Ricci curvature.

\subsection*{Acknowledgments}  This article contains material from the author's PhD Thesis at Imperial College, founded by the European Research Council Grant 247331 and defended in December 2015. I wish to thank my supervisor, Simon Donaldson, for sharing his ideas with me. I also want to thank Song Sun and Cristiano Spotti for valuable conversations, and the Simons Center for Geometry and Physics for hosting me during their program on K\"ahler Geometry during October-November 2015.

\section{Background} \label{back}
Most of this section reviews well known material. In Subsection \ref{SPH MET} we recall the theory of spherical metrics on the projective line that we use. Subsection \ref{KE metrics} is about K\"ahler-Einstein metrics with cone singularities along a divisor.  Subsection \ref{K cones} collects standard facts on Riemannian cones which are also K\"ahler. Finally, Subsection \ref{RFKCcs} introduces the concept of Ricci-flat K\"ahler cone metrics with cone singularities.

\subsection{Spherical metrics with cone singularities on $\mathbb{CP}^1$ } \label{SPH MET}

Fix $ 0 < \beta <1$; on $\mathbb{R}^2 \setminus \lbrace 0 \rbrace$ with polar coordinates $(\rho, \theta)$ let  
\begin{equation}
g_{\beta} = d\rho^2 +  \beta^2 \rho^2 d\theta^2 ,
\end{equation}
this is the metric of a cone of total angle $2\pi \beta$. The apex of the cone is located at $0$ and $g_{\beta}$ is singular at this point.
The metric induces a complex structure on the punctured plane, given by an anti-clockwise rotation of angle $\pi/2$ with respect to $g_{\beta}$; a basic fact is that we can change coordinates so that this complex structure extends smoothly over the origin. Indeed,  setting
\begin{equation}
z = \rho^{1/\beta} e^{i\theta} \label{CHCO}
\end{equation}
we get
\begin{equation} \label{MODEL2}
g_{\beta} = \beta^2 |z|^{2\beta -2} |dz|^2 .
\end{equation}
We denote by $\mathbb{C}_{\beta}$  the complex plane endowed with the singular metric \ref{MODEL2}. 

Consider a Riemann surface $\Sigma$, a point $p \in \Sigma$ and a compatible metric $g$ on $\Sigma \setminus \{p\}$. 

\begin{definition} \label{TDEFINITION} (\cite{troyanov2})
We say that $g$ has cone angle $2\pi\beta$ at $p$ if for any holomorphic coordinate $z$ centered at $p$ we have that 
\begin{equation*}
g = e^{2u} |z|^{2\beta -2} |dz|^2 ;
\end{equation*}
with $u$ a smooth function in a punctured neighborhood of the origin which extends \emph{continuously} over $0$.
\end{definition}

There is an obvious extension of Definition \ref{TDEFINITION}  to the case of finitely many conical points. We are interested in the situation where $\Sigma = \mathbb{CP}^1$ and $g$ has  constant Gaussian curvature $1$ outside the singularities. In our state of affairs we can proceed more directly, giving a local model for the metric around the conical points. 

From now on we set $S^n = \{(x_1, \ldots, x_{n+1}) \in \mathbb{R}^{n+1}: \hspace{1mm} \sum_{i=1}^{n+1} x_i^2 =1\}$  the $n$-sphere, thought as a manifold; we write $S^n(1)$ for the n-sphere with its inherited round metric of constant sectional curvature $1$. 
Let $W$ be a wedge in $S^2(1)$ defined by two geodesics that intersect with angle $\pi\beta$. A local model for a spherical metric with a cone singularity  is  given by identifying two copies of $W$ isometrically along their boundary. The expression of this metric in geodesic polar coordinates $(\rho, \theta)$ centered at the singular point is 

\begin{equation} \label{spherical1}
d\rho^2 + \beta^2 \sin^2(\rho) d\theta^2.
\end{equation}
If we set $ \eta = \left( \tan (\rho/2) \right)^{1/\beta} e^{i\theta}$, our model metric writes as

\begin{equation} \label{spherical2}
4\beta^2 \frac{|\eta|^{2\beta-2}}{(1+|\eta|^{2\beta})^2} |d\eta|^2 .
\end{equation}

Let $p_1, \ldots, p_d$ be $d$ distinct points in $S^2$ and let $\beta_j \in (0, 1)$ for $j=1, \ldots, d$. We say that $g$ is a spherical metric on $S^2$ with cone singularities of angle $2\pi \beta_j$ at the points $p_j$ if $g$ is locally isometric to $S^2(1)$ in the complement of the $d$ points and around each point $p_j$ we can find polar coordinates in which $g$ agrees with \ref{spherical1} with $\beta= \beta_j$.
It follows from what we have said that any such a metric $g$ endows $S^2$ with the complex structure of the projective line with $d$ marked points which record the cone singularities. The correspondence which associates to a spherical metric on $S^2$ a configuration of points in the projective line is the key to the classification of the former, as Theorem \ref{TLT} below shows. Starting from the complex point of view we have the following:

\begin{definition}
	Let $ L_1, \ldots, L_d \in \mathbb{CP}^1$ be $d$ distinct points and $\beta_j \in (0,1)$ for $j=1, \ldots, d$. We say that $g$ is a compatible spherical metric  on $\mathbb{CP}^1$ with  cone singularities of angle $2\pi\beta_j$ at the  points $L_j$ if 
	$g$ is a compatible metric on  $\mathbb{CP}^1 \setminus \lbrace L_1, \ldots, L_d \rbrace$ of constant Gaussian curvature equal to $1$  and  around each singular point $L_j$ we can find a complex coordinate $\eta$ centered at the point in which $g$ is given by $\ref{spherical2}$ with $\beta=\beta_j$. 
\end{definition}

\begin{remark} \label{regularity sph met}
It is equivalent to say that $g$ has cone angle $2\pi \beta_j$ at the points $L_j$, in the sense of Definition \ref{TDEFINITION}, and constant Gaussian curvature $1$  on  $\mathbb{CP}^1 \setminus \lbrace L_1, \ldots, L_d \rbrace$. This equivalence is a consequence of the following local regularity statement: If $g$ is a compatible metric on a punctured disc $D \setminus \{0\} \subset \mathbb{C}$ of constant Gaussian curvature $1$ and cone angle $2\pi \beta$ at $0$; then there is a holomorphic change of coordinates around the origin in which $g$ agrees with \ref{spherical2}.  
\end{remark}

\begin{example}
	The simplest example is when $d=2$, by means of a M\"obius map we can assume that the cone singularities are located at $0$ and $\infty$. The expression \ref{spherical2} globally defines a spherical metric with cone angle $2\pi\beta$ at the given points, this space is also known as the `rugby ball'. It was shown by Troyanov \cite{troyanov1} that \ref{spherical2} is the only compatible spherical metric with cone angle $2\pi\beta$ at $0$ and $\infty$; a consequence of his work is that there are no such metrics with two cone singularities and different cone angle, in particular there can't be a single conical point.
\end{example}

\begin{example}
	We can construct a spherical metric $g$ with three cone singularities of angles   $2\pi \beta_1, 2\pi\beta_2$ and $2\pi\beta_3$ by doubling  a spherical triangle with interior angles $\pi \beta_1, \pi\beta_2$ and $\pi\beta_3$. It follows from elementary spherical trigonometry  that such a triangle $T$ exists and is unique up to isometry if and only if the following two conditions hold: 
	\begin{itemize}
		\item
		\begin{equation} \label{T1}
		\sum_{j=1}^3 \beta_j > 1.
		\end{equation}
		Indeed, the area of $T$ is equal to $\pi(\sum_{j=1}^3 \beta_j -1)$. 
		
		\item
		\begin{equation} \label{T2}
		1-\beta_i < \sum_{j \not= i} (1-\beta_j) \hspace{3mm} \mbox{for} \hspace{2mm}i=1,2,3.
		\end{equation}
		This is the triangle inequality applied to the polar  of $T$.
	\end{itemize}
	In complex coordinates the metric $g$ writes as $g=e^{2u}|dz|^2$ where $u$ is a real function of the complex variable $z$ and, by means of a M\"obius map, we can assume that the cone singularities are located at $0, 1$ and $\infty$. The metric $g$ has an obvious symmetry given by switching the two copies of $T$, which means that $u$ is invariant under the map $z \to \overline{z}$ and is determined by its restriction to the upper half plane. By means of stereographic projection we can think of the triangle $T$ as lying on the complex plane. Let $w=\Phi(z)$ be a Riemann mapping from the upper half plane to $T$, it is then clear that $g$ is the pullback of the standard round metric 
	$$ \frac{4}{(1+|w|^2)^2} |dw|^2$$
	by $\Phi$. It is a classical fact that such a map $\Phi$ is given as the quotient of two linearly independent solutions of the hypergeometric equation
	\begin{equation}
	z(1-z)w''+(c-(a+b+1)z)w'-abw=0
	\end{equation}
	with $\beta_1=1-c$, $\beta_2=a-b$ and $\beta_3=c-a-b.$
	
\end{example}

\begin{example}
	More generally we can consider a spherical convex polygon $P$ with $d$ edges and interior anngles $\pi\beta_1, \ldots, \pi\beta_d$; double it to obtain a spherical metric on $\mathbb{CP}^1$ with cone singularities at some points $L_1, \ldots, L_d$. These points are fixed by the symmetry that switches the two copies of $P$ and this implies that, up to a M\"obius map, we can assume that  the points $L_1, \ldots, L_d$ lie on the real axis. Same as before, the metric in complex coordinates is given by the pullback of the spherical metric by a Riemann mapping from the upper half plane to the polygon. When $d \geq 4$ most spherical metrics are not  doublings  of spherical polygons.
\end{example}

It is a fact that every spherical metric with cone singularities on $S^2$ is isometric to the boundary of a convex polytope inside $S^3(1)$, uniquely determined up to isometries of the ambient space; this includes the doubles of spherical polygons as degenerate cases where all the vertices of the polytope lie on a totally geodesic 2-sphere. Assume that $d\geq 3$; by means of a triangulation and  formulas \ref{T1} and \ref{T2} it is straightforward to show that a necessary condition for the existence of such a metric is that the Troyanov condition \ref{Tcondition} holds. 

Recall that $c$ denotes the number given by \ref{number c}, so that $2c = 2-d + \sum_{j=1}^d  \beta_j$.
By means of a triangulation and the formula for the area of a spherical triangle, it is easy to show that the total area of a spherical metric is given by $4\pi c$.
In algebro-geometric terms, the Troyanov condition is equivalent to say that the pair $(\mathbb{CP}^1, \sum_{j=1}^d (1-\beta_j) L_j)$ is log-K-polystable (see \cite{chili}) and the number $2c$ is the degree of the $\mathbb{R}$-divisor $-(K_{\mathbb{CP}^1} + \sum_{j=1}^d (1-\beta_j) L_j)$. The main result we want to recall is the following:

\begin{theorem} \label{TLT} (Troyanov \cite{troyanov2}, Luo-Tian \cite{luotian})
	Assume that $d\geq 3$, let $L_1, \ldots, L_d$ be $d$ distinct points in $\mathbb{CP}^1$ and let $\beta_j \in (0, 1)$ for $j=1, \ldots, d$. If the Troyanov condition \ref{Tcondition} holds, then there is a unique compatible spherical metric $g$ on $\mathbb{CP}^1$ with cone singularities of angle $2\pi\beta_j$ at the points $L_j$ for $1\leq j \leq d$. 
\end{theorem}

\begin{remark}
	It is an easy consequence of the uniqueness part, that the set of orientation preserving isometries of $g$ agrees with the set of M\"obius maps $F$, which preserve the set $ \{L_1, \ldots, L_d\} $ and such that $F(L_i) = L_j$ only if $\beta_i = \beta_j$.
\end{remark}

Assume that  $ L_j = \lbrace z_1 = a_j z_2 \rbrace$ with $a_j \in \mathbb{C}$  for $j = 1, \ldots , d-1$ and $ L_d= \lbrace z_2=0 \rbrace$. Set $\xi = z_1/z_2$, then $ g = e^{2\phi} |d\xi|^2$ with $\phi$ a function of $\xi$. Recall that the Gaussian curvature of $g$ is given by

\begin{equation*}
K_g = -e^{-2\phi} \triangle \phi , 
\end{equation*}
where $\triangle = 4 \partial^2 / \partial \xi \partial \overline{\xi}$. Then Theorem \ref{TLT} is equivalent to the following statement: Let $a_1, \ldots, a_{d-1} \in \mathbb{C}$ and $\beta_1, \ldots, \beta_d \in (0, 1)$ satisfy the Troyanov condition \ref{Tcondition}; then there exists a unique function $\phi$ such that 

\begin{itemize}
	\item 
	Solves the Liuville equation 
	\begin{equation*}
	\triangle \phi = -e^{2\phi}, 
	\end{equation*}
	in $\mathbb{C} \setminus \{a_1, \ldots, a_{d-1}\}$. 
	\item
	\[ u = \phi - \sum_{j=1}^{d-1} (\beta_j -1) \log |\xi - a_j| \]
	is a continuous function in $\mathbb{C}$, and
	\[ \phi + (\beta_d +1) \log |\xi| \]
	is continuous at $\infty$.
	
\end{itemize}

Let us fix $\beta_1, \ldots, \beta_d \in (0, 1)$ satisfying the Troyanov condition \ref{Tcondition}.
Set $\tilde{\mathcal{P}}_d= \tilde{\mathcal{P}}_d (\beta_1, \ldots, \beta_d)$ to be the space of all boundaries of labeled $d$-vertex convex polytopes in the round  3-sphere with total angle of $2\pi\beta_j$ at $d$ distinct vertices,  modulo the ambient isometries; \ref{Tcondition} ensures that $\tilde{\mathcal{P}}_d$ is not empty. The space $\tilde{\mathcal{P}}_d$ is endowed with the Hausdorff topology.  Let $\tilde{\mathcal{M}}_d$ be the space of $d$ distinct ordered points in $\mathbb{CP}^1$ modulo the action of M\"obius transformations, this is a complex manifold of dimension $d-3$. 
Each element of $\tilde{\mathcal{P}}_d$ represents a spherical metric on $\mathbb{CP}^1$ with cone angle $2\pi \beta_j$ at $d$ distinct points. There is a natural map $\Pi : \tilde{\mathcal{P}}_d \to \tilde{\mathcal{M}}_d$ obtained  by recording the complex structure given by the metric. It is shown in \cite{luotian} that $\Pi$ is a homeomorphism.

\begin{figure}[ht!]
	\centering
	\includegraphics[width=150mm]{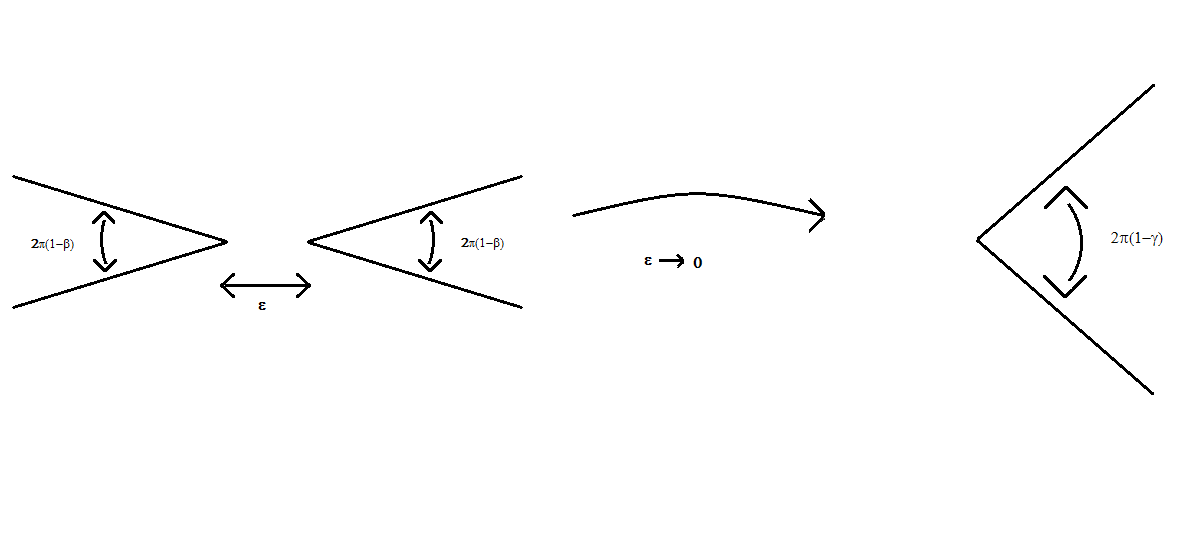}
	\caption{$1 - \gamma = 2 (1-\beta)$. When two cone singularities collide the complements of the angles add.}
	\label{Pacman}
\end{figure}

Consider the case when $\beta_1 = \beta_2 = \ldots = \beta_d = \beta$. Denote by $\mathcal{P}_d$ and $\mathcal{M}_d$ the quotients of $\tilde{\mathcal{P}}_d$ and $\tilde{\mathcal{M}}_d$ by the permutation group on $d$ elements, which corresponds to forgetting the labels. We have an induced homeomorphism $\Pi : \mathcal{P}_d \to \mathcal{M}_d$. 
The Hausdorff topology gives a natural compactification of $\mathcal{P}_d$, similarly the space $\mathcal{M}_d$ has a natural GIT compactification; it is then natural to ask whether $\Pi$ extends as an homeomorphism between these. A useful fact, established in \cite{luotian}, is that the Hausdorff limit of a sequence in $\mathcal{P}_d$ is the boundary of a spherical convex polytope with at most $d$ vertices. We look at  the simplest non-trivial case when $d=4$; the Troyanov condition \ref{Tcondition} is then equivalent to $1/2 < \beta < 1$. The space $\mathcal{M}_4$ of four unordered points on the Riemann sphere is isomorphic to $\mathbb{C}$ and the GIT compactification is isomorphic to the projective line, the extra point added  represents the configuration of two points counted with multiplicity two -this is the unique polystable point-. On the other hand if two cone singularities of angle $2\pi\beta$ collide; what remains is a single cone singularity of angle $2\pi \gamma$ with $\gamma= 2\beta-1$, see Figure \ref{Pacman} and \cite{PM}. Note that there is not any spherical triangle with angles $\pi \gamma, \pi \beta, \pi \beta$, since \ref{T2} would then imply that $1-\gamma = 2- 2\beta < (1- \beta) + (1-\beta)$. We conclude that if two of the vertices collide the other two must collide too; the Hausdorff compactification is obtained by adding a single point, represented by the `rugby ball' with two cone singularities of angle $2\pi \gamma$, which corresponds to the polystable configuration of two points in the projective line with multiplicity two.

\subsection{K\"ahler-Einstein  metrics with cone singularities along a divisor (KEcs)} \label{KE metrics}

We are concerned with metrics which are modeled, in transverse directions to a smooth divisor, by $g_{\beta}$. To begin with we take the product $\mathbb{C}_{\beta} \times \mathbb{C}^{n-1}$; if $(z_1, \ldots, z_n)$ are standard complex coordinates on $\mathbb{C}^n$  what we get is the model metric 
\begin{equation}
	g_{(\beta)} =  \beta^2 |z_1|^{2\beta -2} |dz_1|^2  + \sum_{j=2}^n |dz_j|^2 ,
\end{equation}
with a singularity along $D= \lbrace z_1 =0 \rbrace$.
Set $\lbrace v_1, \ldots, v_n \rbrace$ to be the  vectors 
\begin{equation} \label{VECTORS}
	v_1 = |z_1|^{1-\beta} \frac{\partial}{\partial z_1}, \hspace{3mm} v_j = \frac{\partial}{\partial z_j} \hspace{2mm} \mbox{for} \hspace{1mm} j = 2, \ldots n . 
\end{equation}
Note that, with respect to $g_{(\beta)}$, these vectors are orthogonal  and their length is constant.
We move on and consider the situation of a complex manifold $X$ of complex dimension $n$ and a smooth divisor $ D \subset X$.  Let $g$ be a smooth K\"ahler metric on $X \setminus D$ and let $p \in D$. Take  $(z_1, \ldots, z_n)$ to  be complex coordinates centered at $p$ such that $D = \lbrace z_1 =0 \rbrace$. In the complement of $D$ we have smooth functions $g_{i\overline{j}}$ given by $g_{i\overline{j}} = g (v_i, \overline{v}_j)$. Following Donaldson \cite{donaldson1} we give a definition of a K\"ahler metric with cone singularities which is well suited for the development of a Fredholm theory linearizing  the KE equation:
\begin{definition} \label{DEFINITION}
	We say that $g$ has cone angle $2\pi\beta$ along $D$
	if, for every $p \in D$ and holomorphic coordinates as above, the functions  $g_{i\overline{j}}$ admit a  H\"older continuous extension to $D$. We also require the matrix $(g_{i\overline{j}} (p))$ to be positive definite  and that  $g_{1\overline{j}} =0$  when $j \geq 2$ and $z_1=0$.
\end{definition}

It is straightforward to check that this definition is independent of the holomorphic chart $z_1, \ldots, z_n$. There is a K\"ahler potential $\phi \in C^{2, \alpha, \beta}$ (see \cite{donaldson1}) for $g$ around points of $D$. It can be shown that the vanishing condition,  $g_{1\overline{j}} =0$  for $j \geq 2$ at $z_1=0$, is a consequence of the other conditions; this is related to the behavior of the Green's function for the Laplacian of $g_{(\beta)}$ -see \cite{donaldson1}- and, more geometrically, to the fact that $g_{(\beta)}$ has non-trivial holonomy along simple loops that go around $\{ z_1 = 0\}$. The tangent cone of $g$ at points of $D$ is $\mathbb{C}_{\beta} \times \mathbb{C}^{n-1}$ and its K\"ahler form defines a co-homology class in $X$. 

There are two types of coordinates we can consider around $D$: The first  is given by holomorphic coordinates $z_1, \ldots, z_n$ in which $D= \lbrace z_1 =0 \rbrace$ as before, in the second one we replace the coordinate $z_1$ with $\rho e^{i\theta}$, with $ \rho = |z_1|^{\beta}$ and $e^{i\theta} = \arg (z_1) $, and leave $z_2, \ldots, z_n$ unchanged; we refer to the later as cone coordinates. In other words,  there are two relevant differential structures on $X$ in our situation: One is given by the complex manifold structure we started with, the other  is given by declaring the cone coordinates to be smooth. The two structures are clearly equivalent, by a map modeled on 
\begin{equation*}
(\rho e^{i\theta}, z_2, \ldots, z_n) \to (\rho^{1/\beta} e^{i\theta}, z_2, \ldots, z_n)
\end{equation*}
in a neighborhood of $D$. Note that the notion of a function being H\"older continuous (without specifying the exponent) is independent of the coordinates that we use.

It is easy to come up with examples of metrics which satisfy Definition \ref{DEFINITION}. Indeed,  let $F$ be a smooth positive function and let $\eta$ be a smooth K\"ahler form, both defined on a domain in $\mathbb{C}^n$ which contains the origin. Consider the $(1,1)$ form
\begin{equation} \label{LOC MET}
	\omega = \eta +   i \partial \overline{\partial} (F |z_1|^{2\beta}) .
\end{equation}
Straightforward calculation shows that, in a small neighborhood of $0$, $g$ defines a K\"ahler metric with cone angle $2\pi\beta$ along $D= \{ z_1 =0 \}$. More globally; if $\eta$ is a K\"ahler form on a compact complex manifold $X$, $D \subset X$ is a smooth divisor with a defining section $s \in H^0([D])$, $\epsilon > 0$ is sufficiently small and $h$ is a Hermitian metric on $[D]$. Then 

\begin{equation*}
\omega = \eta +   i \epsilon \partial \overline{\partial} |s|_h^{2\beta} 
\end{equation*}
defines a K\"ahler metric on $X$ with cone angle $2\pi\beta$ along $D$ in the same co-homology class as $\eta$.

We are mainly interested in K\"ahler-Einstein metrics with cone angle $2\pi\beta$ along $D$ (KEcs). These are metrics with cone singularities, as in Definition \ref{DEFINITION},  such that the Ricci tensor is a constant multiple of the metric,
\begin{equation}
	\mbox{Ric}(g_{KE}) = \lambda g_{KE} ,
\end{equation}
in the complement of $D$. From now on we assume that $X$ is compact; among the many results in this area we want to recall the following ones:

\begin{itemize}
	\item \emph{Existence Theory} (\cite{brendle}, \cite{JMR}, \cite{guenancia}).
	The existence results for KEcs parallel the well-known theorems regarding the Calabi conjecture in the case of smooth metrics: 
	$i)$ If $c_1 (X) - (1-\beta)c_1([D]) <0$, then there exists a unique KEcs with $\lambda=-1$. $ii)$ If $c_1 (X) - (1-\beta)c_1([D]) =0$; then, in any K\"ahler class on $X$, there exists a unique KEcs with $\lambda =0$. $iii)$ If  $c_1 (X) - (1-\beta)c_1([D]) >0$ and the \emph{twisted K-energy} is proper, then there exists a unique (up to biholomorphisms which preserve $D$) KEcs  with $\lambda=1$.
	
	\item \emph{Regularity Theory} (\cite{JMR}, \cite{YZ}, \cite{chenwangregularity}).
	The result we want to refer to says that KEcs are `polyhomogeneous'. Let $p \in D$ and $(z_1, \ldots, z_n)$  holomorphic coordinates centered at $p$ in which $D = \{ z_1 =0 \}$. We write $z_1= \rho^{1/\beta} e^{i\theta}$ and denote by $y = (z_2, \ldots, z_n)$ the other coordinate functions.
	Let $g_{KE}$ be a K\"ahler-Einstein metric on $X$ with cone angle $2\pi\beta$ along $D$ and $\beta \in (1/2, 1)$, write $\omega_{KE}$ for the associated K\"ahler form. The regularity theorem says that for every $p \in D$ we can find holomorphic coordinates $(z_1, \ldots, z_n)$ as above such that $\omega_{KE} = i \partial \overline{\partial} \phi$, with
	\begin{equation} \label{ASYM EXP}
		\phi = a_{0}(y) + (a_{01}(y) \cos (\theta) + a_{10} (y) \sin(\theta) ) \rho^{1/\beta} + a_{2}(y) \rho^2 + O(\rho^{2+\epsilon}) . 
	\end{equation}
	Where $a_0, a_{01}, a_{10}, a_2$ are  smooth functions of $y$ and $\epsilon = \epsilon(\beta) >0$.
	When $ \beta \in (0, 1/2]$ the same statement holds if we replace $1/\beta$ with $2$ in the expansion \ref{ASYM EXP}. 
	
	In a different direction, there are results -see \cite{chenwangregularity}- which guarantee that \emph{weak} KEcs are indeed metrics with cone singularities in a H\"older sense, as in Definition \ref{DEFINITION}.

	\item \emph{Chern-Weil formulae} (\cite{SW}, \cite{AL}, \cite{LS}). As shown in the paper of Song-Wang \cite{SW},  the polyhomogeneous expansion implies that the norm of the Riemann curvature tensor of a KE metric with cone angle $2\pi\beta$ is bounded by $\rho^{1/ \beta -2}$. 
	The energy of such a metric $g$ is defined to be
	$$ E(g) = \frac{1}{8\pi^2} \int_X | \mbox{Rm} (g)|^2 =  \frac{1}{8\pi^2}\lim_{\epsilon \to 0} \int_{X\setminus U_{\epsilon}} | \mbox{Rm} (g)|^2  , $$
	where $U_{\epsilon}$ is a tubular neighborhood of $D$ of radius $\epsilon$, $\mbox{Rm}(g)$ denotes the Riemann curvature tensor of $g$ and we integrate using the volume form defined by $g$. It follows that $E(g)$ is finite by comparison with the integral $\int_0^1 \rho^{2/\beta -3} d\rho < \infty$. There is a topological formula for the energy which can be compared with the Chern-Weil formulae in \cite{KrM} for connections with cone singularities.
	This formula expresses the energy of a KE metric of cone angle $2\pi\beta$ along $D$ in terms of $c_1(X), c_2(X), \beta, c_1([D])$ and the cohomology class of the K\"ahler form. When the complex dimension of $X$ is equal to two, the formula reduces to
	\begin{equation}
		E(g_{KE}) = \chi (X) + (\beta -1) \chi (D) .
	\end{equation}

	\item \emph{Compactness Theorem} (\cite{CDS}). 
	Let $X_i$ be a sequence of smooth Fano manifolds with a fixed Hilbert polynomial and  $D_i \subset X_i$  smooth divisors with $ D_i \in |\lambda K_{X_i}^{-1}|$ for some fixed rational number $\lambda \geq 1$ . Fix $ 1-\lambda^{-1} < \beta < 1$. Assume that there exist KE metrics $g_i$ on $X_i$ with cone angle $2\pi\beta$ along $D_i$, we normalize so that $\mbox{Ric}(g_i) = \mu g_i$, with $\mu = 1-(1-\beta)\lambda > 0$. (This normalization condition on the metrics $g_i$ allow us to think of their respective K\"ahler forms as the curvatures of  correponding (singular) Hermitian metrics on $K_X^{-1}$). Approximating the metrics $g_i$ by smooth metrics with a uniform lower bound on the Ricci curvature and a uniform upper bound on the diameter (see \cite{CDS1}) and appealing to the standard Gromov's compactness theorem; shows that  there is, taking a subsequence if necessary, a Gromov-Hausdorff limit of the sequence $g_i$. The main result is then:
	
	\begin{theorem} (Chen-Donaldson-Sun \cite{CDS})  
		There is a $\mathbb{Q}$-Fano variety $W$ and a Weil divisor $\Delta \subset W$ such that:
		\begin{itemize}
			\item The pair $(W, (1-\beta) \Delta)$ is KLT (Kawamata log terminal).
			\item There is a weak conical KE metric for the triple $(W, \Delta, \beta)$ which induces a distance $d$ on $W$; and such that $(W, d)$ is isometric to the Gromov-Hausdorff limit of $(X_i, g_i)$.
			\item There is $m \in \mathbb{N}$ with the property that, up to a subsequence,  we have embeddings $T_i : X_i \to \mathbb{CP}^N$  and $T: W \to \mathbb{CP}^N$ defined by the complete linear systems $H^0(-mK_{X_i})$ and $H^0(-mK_W)$ such that  $T_i (X_i)$ converges to $T(W)$ as algebraic varieties  and $T_i (D_i) \to T(\Delta)$ as algebraic cycles.  
		\end{itemize}
	\end{theorem}
	We won't spell the algebraic geometry words necessary to explain what a KLT pair is, we limit ourselves to say a couple of things in the case of two complex dimensions: 
	
	\begin{enumerate}
		\item The surface $W$ has only finitely many singularities of orbifold type; $\Delta$ is union of irreducible curves counted with multiplicity.
		
		\item Let $p$ be a point in the smooth locus of $W$ which is a singular point  in a component of multiplicity $1$ of the curve $\Delta$, in coordinates centered at $p$ write $\Delta = \{ f=0\}$ for a defining function $f$ with an isolated singularity at $0$; then  $|f|^{2\beta-2}$ is locally integrable. 
	\end{enumerate}

	Similarly, we don't need to write in detail the definition of a weak conical K\"ahler-Einstein metric -see \cite{EGZ}- but we just say that in the complement of $\Delta$ it is a smooth orbifold K\"ahler-Einstein metric.  At points which belong to the smooth locus of multiplicity $1$ components of $\Delta$ the metric cone singularities, in the sense of Definition \ref{DEFINITION}, of cone angle $2\pi\beta$ (Theorem 2 in \cite{CDS}). On the other hand, at smooth points of $\Delta$ of multiplicity $k$ the metric has cone angle 
	\begin{equation} \label{gamma}
	\gamma = k \beta +1-k ,
	\end{equation}
	in the sense that the tangent cone at the point is $\mathbb{C_{\gamma}} \times \mathbb{C}$ (Proposition 13 in \cite{CDS}). It is perhaps better to write \ref{gamma}
	in the form $1-\gamma = k (1- \beta)$; the situation is modeled, in a transverse direction to $\Delta$, by $k$ cone singularities colliding -see Figure \ref{Pacman} for the case $k=2$-.
	
	The goal of this article is to construct and classify the possible tangent cones at singular points of $\Delta$. We concentrate at smooth points of $W$, the general case follows by taking finite coverings.
	For our purposes we can restrict Propositions \ref{Hfm}, \ref{Sfm} and  \ref{Clas} to the situation when $\beta_i= k_i \beta +1-k_i$ for some integers $k_i$; but it is unnatural to add this hypothesis to the statements of our results.
	
\end{itemize}

\subsection{K\"ahler cone metrics} \label{K cones}

A basic reference for this topic is Sparks' survey \cite{sparks}. Let $(S, g_S)$ be a compact Riemannian manifold of real dimension $2n -1$. A Riemannian cone with link $(S, g_S)$ consists of the space $ C = (0, \infty) \times S$  endowed with the metric $g_C = dr^2 + r^2 g_S$, $r$ is the coordinate in the $(0, \infty)$ factor and is then characterized as measuring the intrinsic distance to the apex of the cone; more generally, there is the notion of a metric cone -see \cite{BBI}-. We are particularly interested when $\mbox{Ric} (g_C) \equiv 0$, which is equivalent to $\mbox{Ric}(g_S)= 2(n-1) g_S$. These Ricci-flat cones arise naturally -by means of Bishop-Gromov volume monotonicity theorem- as tangent cones at isolated singularities of limit spaces of non-collapsed sequences of Riemannian manifolds with a lower bound on the Ricci curvature -see \cite{Cheeger}-.

A K\"ahler cone is a Riemannian cone for which there is a parallel complex structure $I$, which makes $C$ into an $n$-dimensional complex manifold. The function $r^2$ is a K\"ahler potential for $g_C$, in the sense that its K\"ahler form is
\[ \omega_C = \frac{i}{2} \partial \overline{\partial} r^2. \]
The Reeb vector field is defined as
\begin{equation*}
\xi = I \left( r \frac{\partial}{\partial r} \right) 
\end{equation*}  
and its flow acts on the cone by holomorphic isometries. 

We restrict our attention to Ricci-flat K\"ahler cones (RFKC), that is K\"ahler cones with $ \mbox{Ric}(g_C)=0$. There is a division RFKC into three types:
\begin{enumerate}
	\item \emph{Regular.} The flow of $\xi$ generates a free $S^1$-action for which the function $r^2$ is a moment map. The K\"ahler quotient of $(C, g_C)$ by this $S^1$-action is an $(n-1)$-dimensional KE Fano manifold, this process can be reverted by means of the so-called Calabi ansatz.
	\item \emph{Quasi-regular.} The flow of $\xi$  generates a locally-free -but not free- $S^1$-action. Same as above, the K\"ahler quotient is a KE Fano orbifold.
	\item \emph{Irregular.} There is at least one non-closed orbit. The closure of the one parameter group generated by $\xi$ is a $k$-dimensional torus, with $k \geq 2$, which acts on the cone by holomorphic isometries.
\end{enumerate}

Let \( (Z, g_{KE}) \) be a normal complex variety with a \emph{weak} KE metric and \(p \in Z \) an isolated singular point. Under suitable circumstances -for example when \( (Z, g_{KE}) \) is the Gromov-Hausdorff limit of a non-collapsed sequence of KE metrics on smooth projective varieties, see \cite{DSII}- there is a unique tangent cone of \(g_{KE}\) at \(p\), this is a Ricci-flat K\"ahler \emph{metric} cone \( (C, g_C) \). The space $C$ is an affine algebraic variety, it is the Spec  of the ring of holomorphic functions on \( C \) of polynomial growth with respect to \(g_C\). Alternatively, $C$ can also be described in terms of a filtration on \( (Z, \mathcal{O}_p) \) -the local ring of regular functions of $Z$ at \(p\)- induced by \(g_{KE}\). It asked in \cite{DSII} whether it is possible to determine \( (C, g_C) \) only in terms of \( (Z, \mathcal{O}_p) \), and to relate this to a \emph{stability} condition for the singularity. There has been recent progress in the case when a neighborhood of \(p\) is biholomorphic to a neighborhood of the apex of a regular RFKC, see \cite{heinsun}.

\subsection {Ricci-flat K\"ahler cone metrics with cone singularities (RFKCcs)} \label{RFKCcs} 

The notion of a flat K\"ahler  metric with cone singularities is defined by means of the local model \( g_{(\beta)} \)

\begin{definition} \label{fmd}
	Let $D \subset X $ be a smooth divisor in a complex manifold $X$. We say that $g$ is a flat K\"ahler metric on $X$ with cone angle $2 \pi \beta$ along $D$; if for every point in the complement of \(D \) we can find holomorphic complex coordinates in which $g$ agrees with the stantard euclidean metric and for every $p \in D$ there are holomorphic coordinates $(z_1, \ldots,  z_n)$ centered at $p$ in which $D = \{ z_1 =0 \}$ and $g$ agrees with $g_{(\beta)}$.
\end{definition}

\begin{example}
	If $\Phi$ is any biholomorphism of $\mathbb{C}^2$ then $\Phi^{*}g_{(\beta)}$ is clearly a flat K\"ahler metric with cone angle $2 \pi \beta$ along $\Phi^{-1}( \{0\} \times \mathbb{C})$. So, if $\Phi (z, w) = ( z- w^2, w)$ then $\Phi^{*} g_{(\beta)}$ has cone angle $2\pi \beta$ along the parabola $z=w^2$.
\end{example}

It should be possible to show, using the polyhomogeneus expansion mentioned in Subsection \ref{KE metrics}, that if \( g \) is a K\"ahler metric on \( X \) with cone angle \( 2 \pi \beta \) along \(D\) -as in Definition \ref{DEFINITION}- which is flat in the complement of \(D\); then it is a flat K\"ahler metric according to Definition \ref{fmd}. Nevertheless, we don't need to use this result.

It is straightforward to combine the notions of Ricci-flat K\"ahler cone (RFKC) and K\"ahler metric with cone singularities to get the following 

\begin{definition}
	Let $D \subset X $ be a smooth divisor in a complex manifold $X$. We say that $g$ is a RFKCcs on \(X\) if it is a RFKC on the complement of \(D\) and it has cone singularities, as in Definition \ref{DEFINITION}, along \(D\).
\end{definition}

When the complex dimension is \(2\), a RFKCcs \(g\) is necessarily flat in the complement of \(D\) and it induces a metric of positive constant Gaussian curvature on its \emph{transversely K\"ahler foliation} -see \cite{sparks}-; Remark \ref{regularity sph met} implies that this is a spherical metric with cone singularities and therefore \(g\) is flat as in Definition \ref{fmd}.

We are mainly interested in the case when \(X= \mathbb{C}^2 \setminus \{0\} \) and \(D\) is a bunch of complex lines which go through the origin and curves of the form \(\{ z^m= a w^n \} \); we allow diferent cone angles at the different components of \(D\). The apex of the cone is at \(0\) and we say that \(g\) is a RFKCcs on \( \mathbb{C}^2 \) -rather than on \(\mathbb{C}^2 \setminus \{0\} \)-.

\begin{example}
	The product $\mathbb{C}_{\beta_1} \times \mathbb{C}_{\beta_2}$  provides an example of a RFKCcs $\mathbb{C}^2$ with cone angle $2\pi \beta_1$ along $\{z_1=0\}$ and $2\pi \beta_2$ along $\{z_2=0\}$. This includes $g_{(\beta)}$ as a particular case when $\beta_2=1$.
\end{example}

According to Panov \cite{panov}, a polyhedral K\"ahler (PK) manifold is a polyhedral manifold whose holonomy is conjugate to a subgroup of \(U(n) \) and every co-dimension \(2\) face with cone angle \(2 \pi k \), \( k \geq 2 \), has a holomorphic direction. If the cone angle at every co-dimension \(2\) face is less than \(2\pi\), then the PK metric is said to be \emph{non-negatively curved}; we restrict to this case. When the PK manifold is also a metric cone, then it is called a PK cone. It is shown in \cite{panov} that the complex structure defined on the complement of the co-dimension \(2\) faces of a PK cone extends; and it defines a RFKCcs on \( \mathbb{C}^2 \). Conversely, any RFKCcs whose link is diffeomorphic to the $3$-sphere is a PK cone.

\section{Spherical  metrics with cone singularities  on the 3-sphere} \label{SECTS}

Let us first describe a local model for a spherical metric in three real dimensions with cone singularities along a codimension two submanifold.
Write $\mathbb{R}^4 = \mathbb{R}^2 \times \mathbb{R}^2$ and take polar coordinates $ (r_1, \theta_1), (r_2, \theta_2)$ on each  factor. Consider the product of a standard cone of total angle $2\pi\beta$ with an Euclidean plane 

\begin{equation}
	g_{(\beta)} = dr_1^2 + \beta^2 r_1^2 d\theta_1^2 +  dr_2^2 +  r_2^2 d\theta_2^2 .
\end{equation}
We want to write $g_{(\beta)}$ as a Riemannian cone; 
it is a general fact that the product of two metric cones is a metric cone. In our case this amounts to check that, if we  define $ r \in (0, \infty) $ and $ \rho  \in (0, \pi/2) $ by 

$$ r_1 = r \sin \rho, \hspace{5mm} r_2 = r \cos \rho ; $$
then  $g_{(\beta)} = dr^2 + r^2  \overline{g}_{(\beta)}$, where

\begin{equation} \label{singsphere}
	\overline{g}_{(\beta)}= d\rho^2 + \beta^2 \sin^2(\rho) d\theta_1^2 + \cos^2(\rho) d\theta_2^2 . 
\end{equation}
We think of $\overline{g}_{(\beta)}$ as a metric on the 3-sphere with a cone singularity of angle $2\pi \beta$ transverse to the circle given by the intersection of $ \lbrace 0 \rbrace \times \mathbb{R}^2$ with $S^3$. It is now straightforward to state the following

\begin{definition} \label{def1}
	Let $S$ be a closed $3$-manifold and let $L \subset S$ be a smooth closed submanifold of codimension two, so that $L= L_1 \cup \ldots \cup L_d$ is a disjoint union of embedded circles $L_j$. Take $\beta_j \in (0,1)$ for $j=1, \ldots, d$. We say that $\overline{g}$ is a spherical metric on $S$ with cone singularities of angle $2\pi \beta_j$ along the $L_j$ if $\overline{g}$ is locally isometric to the round sphere of radius $1$ in the complement of $L$ and around each point of $L_j$ there is a neighborhood in which $\overline{g}$ agrees with $\overline{g}_{(\beta_j)}$
\end{definition}

It shouldn't  be hard to argue that if $S$ admits such a metric, then $S$ must be diffeomorphic to a spherical space form. 

\begin{example} \label{ex2}
	As above, we consider $\mathbb{R}^4 = \mathbb{R}^2 \times  \mathbb{R}^2$ with polar coordinates in each factor. The product of two cones of total angles $2 \pi \beta_1$ and $2 \pi \beta_2$ is given by
	$$ dr_1^2 + \beta_1^2 r_1^2 d\theta_1^2 +  dr_2^2 + \beta_2^2 r_2^2 d\theta_2^2 .$$
	Let $ r \in (0, \infty) $ and $ \rho  \in (0, \pi/2) $ be given by by 
	$ r_1 = r \sin \rho$ and  $r_2 = r \cos \rho;$
	so that the product of the cones writes  $dr^2 + r^2  \overline{g}$, where
	
	\begin{equation}
		\overline{g}= d\rho^2 + \beta_1^2 \sin^2(\rho) d\theta_1^2 + \beta_2^2 \cos^2(\rho) d\theta_2^2 . 
	\end{equation}
	It is easy to check that $\overline{g}$ defines a spherical metric on the 3-sphere with cone singularities of angles $2\pi\beta_1$ and $2\pi\beta_2$ along the Hopf link $L_1 \cup L_2$ given by the intersection of the unit sphere in $\mathbb{R}^4$ with the real planes $\{0\} \times \mathbb{R}^2$ and $\mathbb{R}^2 \times \{0\}$. Nevertheless -unless \( \beta_2 = 1 \)- there is no  neighborhood of $L_1$ isometric to a neighborhood of the singular circle in the model $\overline{g}_{(\beta_1)}$, neither for $L_2$.
\end{example}

\subsection{Hopf bundle}

Let $S^3 = \lbrace |z_1|^2 + |z_2|^2 = 1 \rbrace \subset \mathbb{C}^2$ and consider the Hopf map $H: S^3 \to \mathbb{CP}^1$, $H(z_1, z_2)=[z_1 : z_2]$. This is an $S^1$-bundle with respect to the circle action $e^{it}(z_1, z_2)=(e^{it}z_1, e^{it}z_2)$.
The contraction of the euclidean metric with the derivative of the $S^1$-action gives a $1$-form on $S^3$, referred as the Hopf connection $\alpha_H$. Denote by $g_{FS}$ the Fubini-Study metric on the projective line. By means of stereographic projection $(\mathbb{CP}^1, g_{FS})$ is canonically identified with the round sphere of radius $1/2$ and the Hopf map with 
$$H(z_1, z_2)=(z_1 \overline{z}_2, \frac{|z_2|^2-|z_1|^2}{2}) \in S^2(1/2) \subset \mathbb{R}^3 .$$
It is straightforward to check that the round metric on the 3-sphere is given by
\begin{equation} \label{met3sph}
g_{S^3(1)} = H^{*} g_{FS} + \alpha_H^2 .
\end{equation}
Moreover $d\alpha = H^{*} (\frac{1}{2}K_{FS} dV_{FS})$; where $K_{FS} \equiv 4$ is the Gaussian curvature of the Fubiny-Study metric and $dV_{FS}$ is its area form. 

Let $d\geq2 $,  $L= L_1 \cup \ldots \cup L_d$ be $d$ distinct complex lines going through the origin in $\mathbb{C}^2$ and let $\beta_1, \ldots, \beta_d \in (0,1)$ satisfy the Troyanov condition \ref{Tcondition} ($0<\beta_1=\beta_2<1$ if $d=2$). Denote by $g$ the unique compatible metric on $\mathbb{CP}^1$ of constant Gaussian curvature $4$ and cone angle $2\pi\beta_j$ at the points $L_j$, note that this is $1/4$ times the spherical metric we considered in Subsection \ref{SPH MET}. We shall lift the metric $g$ to a spherical metric on the 3-sphere by means of a suitable connection on the Hopf bundle, in a way analogous to \ref{met3sph}.
We write $K_g$ for the Gaussian curvature of $g$ -which is identically $4$- and $dV_g$ for its area form. The total area of $g$ is $\pi c$ and we write this as a Gauss-Bonnet integral

\begin{equation} \label{GB}
\frac{1}{2\pi} \int_{\mathbb{CP}^1} K_g dV_g =2c . 
\end{equation}

\begin{claim} \label{connection claim}
	There is a connection $\alpha$, unique up to gauge equivalence, such that:
	\begin{enumerate}
		
		\item It has curvature $ d\alpha = (1/2c) H^{*}( K_g dV_g)$.
		
		\item If $p \in \mathbb{CP}^1$ is a point in $L$ and $\gamma_{\epsilon}$ is a loop that shrinks to $p$ as $\epsilon \to 0$, then the holonomy of $\alpha$ along $\gamma_{\epsilon}$ gets trivial as $\epsilon \to 0$.	
	\end{enumerate}		
\end{claim}
We think of $\alpha$ as a $1$-form on $S^3$ singular along $L$. Given a smooth map $f: S^2 \to S^1$, it defines a gauge transformation $\hat{f}: S^3 \to S^3$, $\hat{f}(p)= f(H(p)) \cdot p$; this provides an identification of the group of gauge transformations of the Hopf bundle with the set of maps from the $2$-sphere to $S^1$. Two connections which differ by the pull-back of an exact $1$-form on the base are gauge equivalent. The uniqueness statement in the claim follows from the fact that the first de Rham co-homology group of the punctured $2$-sphere is generated by simple loops which go around the points $L_j$, $j=1, \ldots, d$.  
We prove Claim \ref{connection claim} by writing $\alpha$ explicitly in terms of $g$; before doing this we recall the standard trivializations

\begin{equation} \label{trivialization hopf 1}
S^3 \setminus \{ z_2 \neq 0 \} \cong \mathbb{C} \times S^1, \hspace{2mm} \mbox{given by} \hspace{2mm} (z_1, z_2) \to \left(\xi= \frac{z_1}{z_2}, \hspace{1mm} e^{it}= \arg(z_2)\right) ;
\end{equation}
and 
\begin{equation} \label{trivialization hopf 2}
S^3 \setminus \{ z_1 \neq 0 \} \cong \mathbb{C} \times S^1, \hspace{2mm} \mbox{given by} \hspace{2mm} (z_1, z_2) \to \left(\eta = \frac{z_2}{z_1}, \hspace{1mm} e^{is}= \arg(z_1) \right).  
\end{equation}
These are related via
\begin{equation} \label{change coordinates hopf}
\eta = 1/\xi, \hspace{2mm} e^{is} = \arg(\xi) e^{it}.
\end{equation}
It is easy to write their inverses as
\begin{equation*}
(\xi, e^{it}) \to  \left( z_1= \frac{\xi}{\sqrt{1+|\xi|^2}} e^{it}, \hspace{1mm} z_2= \frac{1}{\sqrt{1+|\xi|^2}} e^{it} \right), \hspace{2mm}
(\eta, e^{is}) \to  \left( z_1= \frac{1}{\sqrt{1+|\eta|^2}} e^{is}, \hspace{1mm} z_2= \frac{\eta}{\sqrt{1+|\eta|^2}} e^{is} \right).
\end{equation*}

We are ready to prove the claim

\begin{proof}
	
	W.l.o.g. we  assume that  $ L_j = \lbrace z_1 = a_j z_2 \rbrace$ with $a_j \in \mathbb{C}$  for $j = 1, \ldots , d-1$ and $ L_d= \lbrace z_2=0 \rbrace$. Set $\xi = z_1/z_2$, then $ g = e^{2\phi} |d\xi|^2$ with $\phi$ a function of $\xi$.
	Set
	
	$$ u = \phi - \sum_{j=1}^{d-1} (\beta_j -1) \log |\xi - a_j| ,$$
	this is a continuous function on $\mathbb{C}$. Moreover
	\begin{equation} \label{derivative of u}
	\lim_{\xi \to a_j}  | \xi - a_j| \frac {\partial u}{\partial \xi}=0
	\end{equation}
	for $j=1, \ldots, d-1$. Indeed, if $\eta$ is a complex coordinate centered at $a_j$ in which
	
	$$ g = \beta^2 \frac{|\eta|^{2\beta -2}}{(1 + |\eta|^{2\beta})^2} |d \eta|^2 .  $$
	Then  $\phi = \log \beta + (\beta -1) \log |\eta| - \log (1 + |\eta|^{2\beta})$ and
	\[  \lim_{\eta \to 0} |\eta|  \frac{\partial}{\partial \eta} \log (1 + |\eta|^{2\beta}) =0 .  \]
	
	On $\mathbb{C} \setminus \lbrace a_1, \ldots , a_{d-1} \rbrace$ define the real 1-form 
	
	\begin{equation}
	\alpha_0 = \frac{i}{2c} (\partial u - \overline{\partial} u) .
	\end{equation}
	It follows from \ref{derivative of u} that,  for $j=1, \ldots, d-1$,

	\begin{equation} \label{gauge condition}
	\lim_{\epsilon \to 0} \int_{C_{\epsilon}(a_j)} \alpha_0 = 0 , 
	\end{equation}
	where $C_{\epsilon}(a_j) = \lbrace |\xi - a_j | = \epsilon \rbrace$.
	On the other hand
	
	\begin{equation} \label{curvature connection}
	d\alpha _0= -\frac{i}{c}\partial \overline{\partial}u = \frac {1}{2c} K_g dV_g , 
	\end{equation}
	so  \ref{GB} gives us that
	
	\begin{equation} \label{integral}
	\frac{1}{2\pi} \int_{\mathbb{C}} d\alpha_0 =1 .
	\end{equation}
	On the trivial $S^1$-bundle $\mathbb{C} \setminus \lbrace a_1, \ldots , a_{d-1} \rbrace \times S^1$ with coordinates $(\xi, e^{it})$ consider the connection $\alpha = dt + \alpha_0$. By means of the trivialization map \ref{trivialization hopf 1} we think of $\alpha$ as a connection on the Hopf bundle. It follows from \ref{gauge condition} and \ref{curvature connection}, that we only need to verify the holonomy condition -second item of the claim- at $\xi = \infty$, which corresponds to the point $L_d$. We use the coordinates \ref{trivialization hopf 2}, where $\eta= 1/\xi$. The change of coordinates \ref{change coordinates hopf} implies that $ \alpha = dt + \alpha_0 = ds + \beta_0 $ with $ \beta_0 =  d(\arg \eta) + \alpha_0$. Now $ \lim_{\epsilon \to 0} \int_{|\eta| = \epsilon} \alpha_0 = - \lim_{N \to \infty} \int_{|\xi| = N} \alpha_0$. It follows from \ref{gauge condition}, \ref{integral} and Stokes' theorem that  $\lim_{N \to \infty} \int_{|\xi| = N} \alpha_0 = 2 \pi$. As a result $ \lim_{\epsilon \to 0} \int_{|\eta |=\epsilon} \beta_0 = 0$.
	
\end{proof}

We proceed with the construction of the spherical metric on the $3$-sphere with cone singularities at the Hopf circles $L$.
There is, at least locally, a description for the model metric $\overline{g}_{(\beta)}$ analogous to \ref{met3sph}. Take polar coordinates $(\rho, \theta)$ on a disc $D$ centered at the origin in $\mathbb{R}^2$ and consider the metric on $D \times S^1$ given by

\begin{equation} \label{loc2}
	d\rho^2 + \beta^2 \frac{\sin^2(2\rho)}{4} d\theta^2 + (d{t} + \beta \sin^2(\rho) d\theta )^2 .
\end{equation}
We claim that \ref{loc2}  is locally isometric to the model $\overline{g}_{(\beta)}$ at the points $\{0\} \times S^1$. Indeed if we set $ t=\theta_2 $ and $\theta = \theta_1 - \theta_2$; then \ref{loc2} writes as  $d\rho^2 + \beta^2 \sin^2(\rho) d\theta_1^2 + \beta^2 \cos^2(\rho) d\theta_2^2$ and this agrees with the metric given in Example \ref{ex2} with $\beta=\beta_1=\beta_2$. We let $\alpha= d{t} + \beta \sin^2(\rho) d\theta$ and think of it as a connection on the trivial bundle $D \times S^1$; it is then easy to check that $d \alpha = (1/2) K_g dV_g$ where $g= d\rho^2 + \beta^2 \frac{\sin^2(2\rho)}{4} d\theta^2$.

\begin{lemma} \label{3sphere}
	There is a -unique up to a bundle isometry- spherical metric $\overline{g}$ on $S^3$ with cone angle $2\pi \beta_j$ along $L_j$ for $j=1, \ldots, d$ such that
	
	\begin{itemize} 
		\item $\overline{g}$ is invariant under the $S^1$ action $e^{it}(z_1, z_2)= (e^{it}z_1, e^{it}z_2)$.
		\item $H : (S^3 \setminus L, \overline{g}) \to (\mathbb{CP}^1 \setminus L, g)$ is a Riemannian submersion with geodesic fibers of constant length.
	\end{itemize}	
\end{lemma}

\begin{proof}	
	Set
	\begin{equation} \label{metric s3}
		\overline{g} = g + c^2 \alpha^2 . 
	\end{equation}
	
	The $S^1$ invariance and Riemannian submersion properties of $\overline{g}$ are evident from its definition. Let us check that $\overline{g}$ is a spherical metric according to Definition \ref{def1}. Let $p  \in S^3$; we use the coordinates \ref{trivialization hopf 1} and \ref{trivialization hopf 2}, w.l.o.g. we assume that $p$ belongs to the domain of definition of the coordinates \ref{trivialization hopf 1} so that $p = (\xi_0, e^{it_0})$.
	There are polar coordinates $(\rho, \theta)$ around $\xi_0$ in which
	
	$$ g = d\rho^2 +  \beta^2 \frac{\sin^2(2\rho)}{4} d\theta^2 ;$$
	where $\beta=1$ if $p \notin L$ and $\beta= \beta_j$ if $p \in L_j$. Write the connection $\alpha = dt + \alpha_0$;
	in these coordinates $d\alpha_0 = (1/2c) K_g dV_g = (1/c) \beta \sin(2\rho) d\rho d\theta$. It follows from the holonomy condition on $\alpha$ that, up to a gauge transformation, we can assume  $\alpha_0 = (1/c) \beta \sin^2(\rho) d\theta$. Then $ c \alpha = c dt + \beta \sin^2(\rho) d\theta$. Finally we take a point distinct from $p$ and on the same fiber, remove it and scale the circle coordinate to obtain the desired expression. More precisely;  if we assume $t_0  \in (-\pi, \pi)$, say,  and define $\overline{t} = c t$ we have that
	
	$$\overline{g} =  d\rho^2 + \beta^2 \frac{\sin^2(2\rho)}{4} d\theta^2 + (d\overline{t} + \beta \sin^2(\rho) d\theta )^2 .$$
	Which agrees with \ref{loc2}.

	Finally, we prove uniqueness. Let $\overline{g}$ be a metric satisfying the conditions of the Lemma. The lengths $l$ of the Hopf circles is constant, write $ l= 2 \pi \tilde{c}$ for some $ \tilde{c} >0$. We obtain a 1-form $\tilde{\alpha}$ by contracting $\overline{g}$ with the derivative of the circle action, then $\tilde{\alpha}= \tilde{c} \alpha$ with $\alpha$ a connection and $\overline{g}= H^{*} g + \tilde{c}^2 \alpha^2$. The fact that at the singular fibers $\overline{g}$ is locally isometric to the models $\overline{g}_{(\beta_j)}$ implies that $\alpha$ must satisfy the holonomy condition; Stokes' Theorem then implies that $(1/2\pi) \int_{\mathbb{CP}^1} d\alpha=1$. The Riemannian submersion property gives us that $d\alpha = (1/2\tilde{c}) K_g dV_g$, therefore $\tilde{c}=c$. The uniqueness then follows from \ref{connection claim}.
\end{proof}

\begin{remark} \label{volume}
	The proof above gives us that the fibers of $H$ have constant length $2\pi c$. Since $\mbox{Vol}(g)=\pi c$ we have $\mbox{Vol}(\overline{g})= 2\pi^2 c^2$. The volume of the round 3-sphere of radius $1$ is $2\pi^2$, so we get that $\mbox{Vol}(\overline{g})/\mbox{Vol}(S^3(1)) = c^2$. This ratio is a relevant quantity in Riemannian convergence theory: If $p$ is a point in a limit space with a tangent cone with link $\overline{g}$; then this volume ratio measures how singular  the limit space is at $p$. Roughly speaking the smaller the volume ratio the worse the singularity.
\end{remark}

Lemma \ref{3sphere} is also established in \cite{panov}; for the sake of completeness we repeat the arguments in \cite{panov}, these make clear why the fibers of $\overline{g}$ must have length $2 \pi c$: If \( \Omega \subset S^2(1/2) \) is a contractible domain; then the universal cover of $ H^{-1}(\Omega) \subset S^3(1)$ is diffeomorphic to \( \Omega \times \mathbb{R} \), and its inherited constant curvature $1$ metric is invariant under translations on the \( \mathbb{R} \) factor. The planes orthogonal to the fibers define a horizontal distribution, hence a connection $\nabla$ on \( \Omega \times \mathbb{R} \). The holonomy of $\nabla$ along a closed curve \( \gamma \subset \Omega \) is equal to the parallel translation by twice the algebraic area bounded by $\gamma$. On the other hand; for any $l >0$ we can take the quotient of \( \Omega \times \mathbb{R} \) by \( l \mathbb{Z} \) to obtain a metric $\overline{g}$ of contant curvature $1$ on \( \Omega \times S^1 \) such that all the fibers are geodesics of length $l$. 
Given the metric $g$ on $S^2$ with cone singularities and Gaussian curvature $4$; we can cut $S^2$ by geodesic segments with vertices at all the conical points and obtain a contractible polygon $P$ which can be immersed -by its enveloping map- in $S^2(1/2)$. Consider the metric $\overline{g}$ on \( P \times S^1 \) 
with $l = 2 \mbox{Area}(P)$. It follows that the holonomy of the fibration along the border of $P$ is trivial -as it makes one full rotation-; and the gluing of $P$ which gives $g$ can be lifted to a gluing of \( P \times S^1 \) to obtain the metric $\overline{g}$ of Lemma \ref{3sphere}.

\subsection{Seifert bundles and branched coverings} \label{seifertsection}
Let $p$ and $q$ be positive co-prime integers, w.l.o.g. we can assume that $1 \leq p < q$.  Consider the $S^1$-action on $S^3= \{ |z_1|^2 + |z_2|^2=1 \} \subset \mathbb{C}^2 $ given by 
\begin{equation} \label{mn}
	e^{it}(z_1, z_2)= (e^{ipt}z_1, e^{iqt}z_2)
\end{equation}
together with the Seifert map $S_{(p,q)}: S^3 \to \mathbb{CP}^1$ given by $S_{(p,q)}(z_1, z_2) = [z_1^q, z_2^p]$. The map $S_{(p, q)}$ is invariant under the $S^1$-action \ref{mn} and restricts to an $S^1$-bundle over $\mathbb{CP}^1 \setminus \{ [1, 0], [0,1] \}$. The fiber of $S_{(p,q)}$ over a point in the projective line distinct from the poles is a torus knot of type $(p, q)$. Around the pole $[1,0]$ there is a disc $U$ and an $S^1$-equivariant diffeomorphism from $S_{(p,q)}^{-1}(U)$ to the solid torus $D \times S^1$ with the $S^1$-action $e^{it} (z, e^{i\theta})= (e^{ipt}z, e^{iqt}e^{i\theta})$, similarly there is a disc around the pole $[0, 1]$ and an $S^1$-equivariant diffeomorphism from the preimage of the disc to the solid torus with $S^1$-action $e^{it} (z, e^{i\theta})= (e^{iqt}z, e^{ipt}e^{i\theta})$. In this section we  lift the spherical metrics on the projective line to the 3-sphere by means of $S_{(p, q)}$, we do this by means of a branched covering map $\Psi_{(p,q)} : S^3 \to S^3$  given by
\begin{equation}
	\Psi_{(p,q)} (z_1, z_2)= \left( \frac{z_1^q}{\sqrt{|z_1|^{2q} + |z_2|^{2p}}},  \frac{z_2^p}{\sqrt{|z_1|^{2q} + |z_2|^{2p}}}  \right) .
\end{equation}
Note that
$S_{(p,q)}= H \circ \Psi_{(p,q)}$.
The map $\Psi_{(p,q)}$ is a branched $pq$-fold cover, branched along the two exceptional fibers of $S_{(p,q)}$. It is equivariant with respect to the circle actions $(e^{ipt}z_1, e^{iqt}z_2)$ and $(e^{ipqt}z_1, e^{ipqt}z_2)$.
\begin{example} \label{seifertexample}
	Let $g$ be the spherical metric on $\mathbb{CP}^1$ with cone angles $2\pi (1/q)$ at $[0,1]$, $2\pi (1/p)$ at $[1,0]$ and $2\pi\beta$ at $[1,1]$.  The triple $(1/q, 1/p, \beta)$ satisfies the Troyanov condition \ref{Tcondition} if and only if $1-1/p-1/q< \beta < 1 -1/p + 1/q$. Write $\overline{g}$ for the lift of $(1/4)g$ to $S^3$ by means of the Hopf map. Set $\tilde{g}= \Psi^{*} \overline{g}$. It is then clear that $\tilde{g}$ is  a spherical metric on $S^3$, invariant under the $S^1$-action \ref{mn} and has a cone singularity of angle $2 \pi \beta$ along the $(p, q)$-torus knot $\{ (e^{ip\theta} / \sqrt{2} , e^{iq\theta} / \sqrt{2}) , \hspace{2mm} \theta \in [0, 2\pi]  \}$.
\end{example}

The following Lemma is an immediate consequence of Lemma \ref{3sphere}, pulling-back the metric $\overline{g}$ with the map $\Psi_{(p,q)}$.

\begin{lemma} \label{slema}
	Set $d\geq 3$. Let $\beta_1, \ldots, \beta_{d-2} \in (0,1)$ and  $\beta_{d-1}, \beta_d \in (0,1]$. Assume that the numbers 
	$$\beta_1, \ldots, \beta_{d-2}, (1/q)\beta_{d-1}, (1/p) \beta_d $$ 
	satisfy the Troyanov condition \ref{Tcondition}. Let $g$ be the spherical metric on $\mathbb{CP}^1$ with cone angles $2 \pi (1/q) \beta_{d-1}$ at $[0,1]$, $2 \pi (1/p) \beta_d$ at $[1,0]$ and $ 2\pi \beta_j$ at $s_j$ for $1 \leq j \leq d-2$. Then there is a spherical metric $\tilde{g}$ on $S^3$ with cone angles $2\pi \beta_{d-1}$ and $2 \pi \beta_d$ at the Hopf circles which fiber over $[0,1]$ and $[1,0]$ respectively; and cone angles $2 \pi \beta_j$ at the $(p, q)$-torus knots which fiber over the points $s_j$, $1 \leq j \leq d-2$. The metric $\tilde{g}$ is invariant under the $S^1$-action \ref{mn} and $ S_{(p, q)} : (S^3, \tilde{g}) \to (\mathbb{CP}^1, (1/4) g)$ is a Riemannian submersion with geodesic fibers of constant length $2 \pi \tilde{c}$ at the non-exceptional orbits. Furthermore, the metric $\tilde{g}$ is unique up to a $S^1$-equivariant isometry which induces the identity on $\mathbb{CP}^1$
\end{lemma}

\section {Flat K\"ahler cone metrics on $\mathbb{C}^2$} \label{SECTC}

\subsection{Proof of Propostion \ref{Hfm}}

Let $L_j = \{ l_j (z,w) =0 \}$ for $j=1, \ldots, d$ be $d$ distinct complex lines through the origin in $\mathbb{C}^2$ with defining linear equations $l_j$. Let $\beta_1, \ldots, \beta_d \in (0,1)$ satisfy the Troyanov condition; $g$ be the spherical metric on $\mathbb{CP}^1$ with cone angle $2 \pi \beta_j$ at $L_j$ and  $\overline{g}$ be the lifted metric on $S^3$ by means of the Hopf map. We set 
\begin{equation}
g_F = dr^2 + r^2 \overline{g} ,
\end{equation}
to be the Riemannian cone with $(S^3, \overline{g})$ as a link, this is a metric on $(0, \infty) \times S^3$ with cone singularities along the products of the singular Hopf circles with the radial coordinate. We shall prove that there is a natural complex structure $I$ with respect to which $g_F$ is K\"ahler; that there is a natural identification of this complex manifold with $\mathbb{C}^2$, with repect to which the Reeb vector field generates the circle action $e^{it}(z,w)=(e^{it/c}z, e^{it/c}w)$  and the singularities of $g_F$ are along the original set of lines $L_j$ we started with.

We use  the same coordinates as in the proof of Lemma \ref{3sphere}, where the Hopf bundle is trivialized; so that points in $\mathbb{R}_{>0} \times (S^3 \setminus L) \cong (0, \infty) \times \mathbb{C} \setminus \lbrace a_1, \ldots, a_{d-1} \rbrace \times S^1$ have coordinates $(r, \xi, e^{it})$. Write $ \xi = x + iy$. Consider the almost-complex structure given by

$$ I \tilde{  \frac{\partial}{\partial x}  } = \tilde{  \frac{\partial}{\partial y}  } , \hspace{15mm}  I   \frac{\partial}{\partial r}   = \frac{1}{cr}   \frac{\partial}{\partial t}      $$
where

$$  \tilde{  \frac{\partial}{\partial x}  } =  \frac{\partial}{\partial x}  - \alpha \left( \frac{\partial}{\partial x} \right) \frac{\partial}{\partial t} , \hspace{15mm} \tilde{  \frac{\partial}{\partial y}  } =  \frac{\partial}{\partial y}  - \alpha \left( \frac{\partial}{\partial y} \right) \frac{\partial}{\partial t}  $$
are the horizontal lifts of $\partial / \partial x$ and $\partial / \partial y$. Finally set $\omega_F = g_F (I. , .)$.

\begin{claim}
	$ \left( (0, \infty) \times \mathbb{C} \setminus \lbrace a_1, \ldots, a_{d-1} \rbrace \times S^1    , g_F, I \right) $ is a K\"ahler manifold. I.e. $d\omega_F =0$ and $I$ is integrable. Moreover, 
	\begin{equation} \label{potential}
	\omega_F = \frac{i}{2} \partial \overline{\partial} r^2 .
	\end{equation}
	
\end{claim}

\begin{proof}
	We compute in the coframe $\lbrace dx, dy, dr, \alpha \rbrace$ where
	
	$$ \omega_F = r^2 e^{2\phi} dx \wedge dy + cr dr \wedge \alpha , $$
	so that $ d\omega_F = 2re^{2\phi}dr dx dy - cr(2/c)e^{2\phi} drdxdy = 0$. The integrability of $I$ amounts to check that
	
	$$ \left[  \tilde{  \frac{\partial}{\partial x}  } + i  \tilde{  \frac{\partial}{\partial y}  },   \frac{\partial}{\partial r}   + i \frac{1}{cr}   \frac{\partial}{\partial t} \right] = 0  . $$
	Finally $dId(r^2) = d (2r I dr)=  - 2c d(r^2 \alpha) = -4cr dr \wedge \alpha - 4 r^2 e^{2\phi} dx \wedge dy$. Using that $2i \partial \overline{\partial} = -dId$ we deduce \ref{potential}

\end{proof}

\begin{claim}
	The functions 
	
	\begin{equation} \label{change coordinates}
	z= \xi w, \hspace{4mm} w = c^{1/2c}  r^{1/c} e^{u/2c}  e^{it}
	\end{equation}
	give a biholomorphism between $ (0, \infty) \times \mathbb{C} \setminus \lbrace a_1, \ldots, a_{d-1} \rbrace \times S^1$ with the complex structure $I$ and $\mathbb{C}^2 \setminus L$. If we write $\Omega = (\sqrt{2})^{-1} dz dw$,  then
	
	\begin{equation} \label{flat volume}
	\omega_F^2 = | l_1 |^{2\beta_1 - 2} \ldots |l_d|^{2\beta_d-2} \Omega \wedge \overline{\Omega} .
	\end{equation}
\end{claim}

\begin{proof}
	It is easy to see that the pair $(z, w)$ defines a diffeomorphism between the corresponding spaces.  The Cauchy-Riemann equations for a function $h$ to be holomorphic with respect to $I$ are given by
	
	$$  \frac{\partial h}{\partial r}   + i \frac{1}{cr}   \frac{\partial h}{\partial t} = 0, \hspace{15mm} \frac{\partial h}{\partial x}   
	+ i \frac{\partial h}{\partial y }  =  \alpha \left( \frac{\partial}{\partial x} + i \frac{\partial}{\partial y} \right) \frac{\partial h}{\partial t} . $$
	If we ask $h$ to have weight 1 with respect to the circle action the equations become
	$$ \frac{\partial h}{\partial r}  =  \frac{1}{cr}h ,  \hspace{15mm} \frac{\partial h}{\partial \overline{\xi}} = i \alpha \left( \frac{\partial}{\partial \overline{\xi}} \right) h = \frac{1}{2c} \frac{\partial u}{\partial \overline{\xi}} h . $$
	It is now easy to check that $z$ and $w$ are holomorphic. 
	
	Now we compute the volume form of $g_F$ in the complex coordinates $z, w$. First define a basis  $\lbrace \tau_1, \tau_2 \rbrace$  of the $(1, 0)$ forms
	
	\begin{equation} \label{orthbasis}
	\tau_1= dr + i cr \alpha , \hspace{4mm}  \tau_2 = e^{\phi} r d\xi  . 
	\end{equation}
	Up to a factor of $\sqrt{2}$ this is an orthonormal basis for the $(1,0)$ forms in $\mathbb{C}^2 \setminus L$,  i.e. 
	
	$$\omega_F = (i/2) \tau_1 \overline{\tau_1} + (i/2) \tau_2 \overline{\tau_2} .$$ 
	Define a two by two matrix $(a_{ij})$ by means of
	
	$$ dz = a_{11} \tau_1 + a_{12} \tau_2, \hspace{4mm} dw = a_{21} \tau_1 + a_{22} \tau_2 . $$
	From here we get
	$$ \Omega \wedge \overline{\Omega} = | \det (a_{ij}) |^2 \omega_F^2 . $$
	Since $ z = \xi w$ we have that $ a_{11} = \xi a_{21}$ and $ a_{12} = \xi a_{22} + w e^{-\phi} r^{-1}$. It follows that $\det (a_{ij}) = -w e^{-\phi} r^{-1} a_{21}$. We can  easily compute,  from the formula given for $w$,  that  $a_{21} = (1/cr) w$. We put things together to get 
	
	$$ \omega_F^2 = c^2  |w|^{-4} r^4 e^{2\phi} \Omega \wedge \overline{\Omega} . $$
	Now we use that $ r^4 = (1/c^2)  |w|^{2c} e^{-2u}$, $\phi - u = \sum_{j=1}^{d-1} (\beta_j -1) \log | (z/w) - a_j |$ and $ 4c -4 = \sum_{j=1}^d (2\beta_j -2)$ to conclude that
	
	$$ \omega_F^2 = |z - a_1 w|^{2\beta_1 -2} \ldots |z - a_{d-1} w|^{2\beta_{d-1} -2} |w|^{2\beta_d -2} \Omega \wedge \overline{\Omega} .$$
	This is formula \ref{flat volume}.
	
\end{proof}

Note that we have two natural systems of coordinates: the complex coordinates $(z,w)$ and the spherical coordinates $(r,\theta)$,  where $\theta$  denotes a point in the 3-sphere. For $\lambda >0$ define $D_{\lambda} (r , \theta) = (\lambda r, \theta)$ and $m_{\lambda} (z, w) = (\lambda z, \lambda w)$.  Equation \ref{change coordinates} gives that $D_{\lambda} = m_{\lambda^{1/c}}$ and Equation \ref{potential} implies that $m_{\lambda}^{*} g_F = \lambda^{2c} g_F$. The proof of the existence part proof of Proposition \ref{Hfm} is now complete. 

We have obtained a  a recipe which allows us  to go from the flat metric $g_F$ on $\mathbb{C}^2$ in Proposition \ref{Hfm} to the corresponding spherical metric $g$ on $\mathbb{CP}^1$  and vice versa.
From \ref{change coordinates} we get

\begin{equation} \label{form1}
r^2 = \frac{1}{c} |w|^{2c} e^{-u} .
\end{equation}
We recall that 
\begin{equation} \label{form2}
u = \phi -   \sum_{j=1}^{d-1} (\beta_j -1) \log |\xi - a_j| ,   \hspace{3mm} g = e^{2\phi} |d\xi|^2 .
\end{equation}
Where $\phi$ a function of $\xi = z/w$. We are writing  the lines as $ L_j = \lbrace z = a_j w \rbrace$ with $a_j \in \mathbb{C}$  for $j = 1, \ldots , d-1$ and $ L_d= \lbrace w=0 \rbrace$. 
\ref{form1} together with \ref{form2} allow us to write $g_F$ explicitly in terms of $g$  and vice-versa. As a check, let us recall the rugby ball metric

\begin{equation} \label{ex1}
g = \beta^2 \frac{|\xi|^{2\beta}}{(1+ |\xi|^{2\beta})^2} |d\xi|^2 , 
\end{equation}
We use our formula \ref{form1} to get $r^2 = \beta^{-2}(|z|^{2\beta} + |w|^{2\beta})$,  so that $g_F = |z|^{2\beta-2} |dz|^2 + |w|^{2\beta-2} |dw|^2$. Up to a constant normalizing factor this is the space $\mathbb{C}_{\beta} \times \mathbb{C}_{\beta}$. 

\begin{remark}
	Since the lenght of any Hopf circle with respect to \( \overline{g} \) is \(2\pi c \); we conclude that the restriction of \(g_F\) to any complex line which goes through the origin, is the metric of a 2-cone with total angle \(2\pi c \).
\end{remark}

The uniqueness statement in Proposition \ref{Hfm} is a consequence of the uniqueness of spherical metrics -Theorem \ref{TLT}-; since given the metric \(g_F\), we can use equations \ref{form1} and \ref{form2} to get the corresponding spherical metric on the projective line.

\subsection{Proof of Proposition \ref{Sfm}}
Let $d \geq 2$, $1 \leq p < q$ co-prime and  $C_j = \{z^q =a_j  w^p \}$, $a_j \in \mathbb{C}$ for $j=1, \ldots, d-2$ be distinct complex curves through the origin in $\mathbb{C}^2$. Let $\beta_1, \ldots, \beta_{d-2} \in (0,1)$ and $0<\beta_{d-1}, \beta_d \leq1$ be such that $\beta_1, \ldots, \beta_{d-2}, (1/q)\beta_{d-1}, (1/p)\beta_d$ satisfy the Troyanov condition \ref{Tcondition} if $ d\geq 3$ and $\beta_{d-1} / q = \beta_d /p$ if $d=2$. In $\mathbb{C}^2$ with complex coordinates $(u,v)$ consider the metric $g_F$  given by Proposition \ref{Hfm} with cone angles $\beta_1, \ldots, \beta_{d-2}, (1/q)\beta_{d-1}, (1/p)\beta_d$ along the lines $L_j = \{u =a_j  v \}$ for $j=1, \ldots, d-2$, $\{u=0\}$ and $\{v=0\}$.  Let $S : \mathbb{C}^2 \to \mathbb{C}^2$ be given by
\begin{equation}
(u, v) = S(z, w)= (z^q, w^p) .
\end{equation}
Proposition \ref{Sfm} follows by setting $\tilde{g}_F = S^{*}g_F$; it is also clear that $\tilde{g}_F$ is isometric to the Riemannian cone with link the metric $\tilde{g}$ given by Lemma \ref{slema}.

As an example, we let $ 1 -1/m -1/n < \beta < 1 -1/m + 1/n$ and set $g_F$ to be the flat metric in $\mathbb{C}^2$ with cone angles $2 \pi (1/n)$ along $\{u=0\}$, $2 \pi (1/m)$ along $\{v=0\}$ and $\beta$ along $\{u=v\}$, then $\tilde{g}_F$ has cone angle along the curve $\{z^n=w^m\}$. In particular we have that for any $1/6< \beta < 5/6$ there is a flat K\"ahler cone metric in $\mathbb{C}^2$ with cone angle $2 \pi \beta$ along the cuspidal cubic $\{w^2=z^3\}$.

\subsection{Proof of Proposition \ref{Clas}}

The proof of Proposition \ref{Clas} is included only for the sake of completeness. We follow the arguments given in Lemma 3.9 and Proposition 3.10 of \cite{panov}; and refer to \cite{panov} for a more detailed exposition.

Let $g_C = dr^2 + r^2 g_S$ be a flat K\"ahler metric  with cone singularities, with link \(S\) diffeomorphic to the $3$-sphere. The fact that \(g_C\) is flat implies that \(g_S\) is spherical. 
There is a orthogonal parallel complex structure $I$ on $(0, \infty) \times S$; the Reeb vector field is 
\begin{equation*}
\xi = I \left( r\frac{\partial}{\partial r}\right) .
\end{equation*}
We think of $S$ as lying inside the cone by means of the isometric embedding which takes $p \in S$ to $(p, 1) \in S \times (0, \infty)$. The restriction of $\xi$ to $S$ is a unit length Killing vector field and its orbits define a one-dimensional foliation of $S$. There are two cases to consider:

\begin{itemize}
	\item All the orbits are periodic. The  flow of $\xi$ defines a locally free $S^1$-action on $S$ by isometries which provides \(S\) with the structure of a Seifert bundle. The classification of Seifert bundles whose total space is the $3$-sphere -see \cite{orlik}-, implies that -up to a conjugation by a diffeomorphism- the $S^1$-action is given by $ e^{it}(z_1, z_2)= (e^{it}z_1, e^{it}z_2)$ if it is free and  $ e^{it}(z_1, z_2)= (e^{imt}z_1, e^{int}z_2)$ with $1 \leq m < n$ for some co-prime numbers $m$ and $n$ if not. We can push \(g_S\) to the quotient to obtain a metric on the $2$-sphere with cone singularities and constant curvature \(4\). The uniqueness statements in Lemma \ref{3sphere} and Lemma \ref{slema}, imply  that $g_S$ must be isometric to one of the metrics $\overline{g}$ of Lemma \ref{3sphere} in the free case or to  one of the metrics $\tilde{g}$ of Lemma \ref{slema} in the locally free but not free case. It follows that \(g_C\) must agree with one of the metric of Propositions \ref{Hfm} or \ref{Sfm}.
	
	\item If there is a non-closed orbit of $\xi$ then there is a $2$-dimensional torus $T^2$ which acts by holomorphic isometries on \(C\). Write \(L\) for the singular locus of \(g_C\) and let \(E\) be the \emph{enveloping map}, which goes from the universal cover of \( C \setminus L \) to \( \mathbb{C}^2 \) and sends the apex of the cone to \(0\). There is an induced action of \( \mathbb{R}^2\) on the euclidean \( \mathbb{C}^2 \) which fixes \(0\) and makes \(E\) equivariant. This action factors through \(T^2\) and we can assume that it is given by rotations on each of the factors \( \mathbb{C} \times \mathbb{C} \). The branching locus of \(E\) is the union of lines through \(0\) invariant by \(T^2\), so it must be the set \( \{z_1 z_2 =0\} \). It follows that
	\[ E: E^{-1} (\mathbb{C}^2 \setminus \{z_1 z_2 =0\}) \to \mathbb{C}^2 \setminus \{z_1 z_2 =0\} \]
	is a covering map; and therefore \(g_C\) is a product of two $2$-cones.
\end{itemize}

\subsection{Hermitian metrics on line bundles: A different approach}

We mention another approach to Proposition \ref{Hfm} which gives the metric in $\mathbb{C}^2$ directly in terms of the metric in the projective line, avoiding to go through the 3-sphere. We take the point of view of a K\"ahler metric as the curvature form of a Hermitian metric on a complex line bundle. We discuss the Hopf bundle case, for the Seifert bundle case there is a parallel discussion in which one replaces the projective line with the weighthed $\mathbb{P}(m,n)$.

We think of $\mathbb{C}^2$ as the total space of  $\mathcal{O}_{\mathbb{CP}^1} (-1)$ with the zero section collapsed at $0$. The bundle projection is given by $\Pi : \mathbb{C}^2 \setminus \lbrace 0 \rbrace \to \mathbb{CP}^1$, $\Pi (z, w) = [z:w]$. We can then identify (smooth) Hermitian metrics on $\mathcal{O}_{\mathbb{CP}^1} (-1)$ with (smooth) functions $ h : \mathbb{C}^2 \to \mathbb{R}_{\geq 0}$ such that $h( \lambda p) = | \lambda |^2 h(p)$ for all $\lambda \in \mathbb{C}$, $ p \in \mathbb{C}^2$ and $h(p) =0$ only when $p=0$.
The first basic fact we need is that an area form $\omega$ in $\mathbb{CP}^1$ induces a Hermitian metric $h_{\omega}$.  We use coordinates $\xi =z/w$, $\eta=w/z$ on $\mathbb{CP}^1$. Write $ \omega= e^{2\phi} (i/2) d\xi d\overline{\xi}$ with $\phi = \phi(\xi)$ on $U =\Pi (\lbrace w\not= 0 \rbrace)$ and $ \omega= e^{2\psi} (i/2) d\eta d\overline{\eta}$ with $\psi= \psi(\eta)$ on $V =\Pi (\lbrace z\not= 0 \rbrace)$. Then $h_{\omega}$ is given by

\begin{equation}
h_{\omega} = |w|^2 e^{-\phi}, \hspace{2mm} \mbox{if} \hspace{2mm} w \not= 0 ; \hspace{4mm} h_{\omega} = |z|^2 e^{-\psi}, \hspace{2mm} \mbox{if} \hspace{2mm} z \not= 0 .
\end{equation}
The second basic fact is that a Hermitian metric $h$ gives a 2-form $\omega_h$ on $\mathbb{CP}^1$ by means of

\begin{equation}
\omega_h = i \partial \overline{\partial} \log h(\xi, 1) \hspace{2mm} \mbox{on} \hspace{2mm} U, \hspace{2mm} \mbox{and} \hspace{2mm} \omega_h = i \partial \overline{\partial} \log h(1, \eta) \hspace{2mm} \mbox{on} \hspace{2mm} V .
\end{equation}
We also mention that $h$ induces Hermitian metrics on the other complex line bundles over $\mathbb{CP}^1$. A linear function $l(z, w)= z -a w$ on $\mathbb{C}^2$  can be regarded as a section of $\mathcal{O}_{\mathbb{CP}^1} (1)$,  then we have $| l|^2_{h} = h(\xi, 1)^{-1} | \xi - a|^2$
on $U$ and a corresponding expression on $V$.

One can then rephrase  the existence of the spherical metric with cone singularities $g$ on $\mathbb{CP}^1$  by saying that there is a Hermitian metric $h$, continuous on $\mathbb{C}^2$ and smooth outside $L$ such that

\begin{equation} \label{csc}
h= | l_1|_{h}^{\beta_1-1} \ldots |l_d|_h^{\beta_d-1} h_{\omega_h}
\end{equation}
Where by $|l|_h$ we mean $|l|_h \circ \Pi$. Here we could be more precise and instead of saying that $h$ is merely continuous we could give a local model for $h$ around the singular points. 
From \ref{csc} one gets that $\omega_h$ has constant Gaussian curvature equal to $2c = 2 -d + \sum_{j=1}^d \beta_j$ outside $L$ and one can  argue that  $(2\pi)^{-1} \int_{\mathbb{CP}^1} \omega_h = 1$.
The potential for $\omega_F$ is then given by $r^2 = a h^{c}$ for some constant $a>0$ determined by the volume normalization.

\subsection{Quotients and Unitary Reflection Groups} \label{URG}

We begin by recalling the well-known Du Val singularities. Let \( \Gamma \subset SU(2)\) be a finite subgroup, up to conjugation, we can assume that it is one of following list:
\(C_m \) -cyclic of order \(m\)- for some \(m \geq 2 \); 
\( D_{2m} \) -binary dihedral of order \(4m\)- for some \(m \geq 2\);
\(T\)  -binary tetrahedral-;
\(O\) -binary octahedral-;
\(I\) -binary icosahedral-.
Basic work of Klein shows that there are three homogeneous polynomials \( z, w, t \in \mathbb{C}[x_1, x_2] \), and \( p \in \mathbb{C}[z, w, t] \), which define a complex isomorphism between the orbit space \( \mathbb{C}^2 / \Gamma \) and the complex surface \( S= \{ p(z, w, t) =0 \} \subset \mathbb{C}^3\). This surface has an isolated singular point at \(0\), referred as a Du Val -or simple- singularity; the list of these is

\begin{itemize}
	\item \(A_m\), \(m \geq 1\): \( S= \{t^2 + w^2 = z^{m+1}\} \), \( \Gamma=C_{m+1}  \).
	\item \(D_m\), \( m \geq 4 \): \( S=\{ t^2 + z w^2 = z^{m-1} \} \), \( \Gamma = D_{2(m-2)} \).
	\item \(E_6\): \( S=\{ t^2 + w^3 = z^4 \} \), \( \Gamma = T \).
	\item \(E_7\): \(S = \{t^2 + w^3= wz^3\} \), \( \Gamma = O \).
	\item \(E_8\): \(S= \{t^2 + w^3 = z^5\} \), \( \Gamma = I \).
\end{itemize}

The map \( S \to \mathbb{C}^2 \) given by \( (z, w, t) \to (z, w) \) is a double cover, branched along the curve \( C  \subset \mathbb{C}^2 \) composed by the points \( (z,w) \) such that \( (z, w, 0) \in S \). This curve has an isolated singularity at the origin, these are the so-called simple plane curve singularities
\begin{equation}
A_m : \hspace{1mm} w^2=z^{m+1}, \hspace{3mm} D_m : \hspace{1mm} zw^2= z^{m-1}, \hspace{3mm}
E_6 : \hspace{1mm} w^3=z^4, \hspace{3mm} E_7 : \hspace{1mm} w^3 = w z^3, \hspace{3mm} E_8 : \hspace{1mm} w^3=z^5 .
\end{equation} 

The group \( \Gamma\) acts freely on \(S^3\), it preserves the round metric \(g_{S^3(1)}\) so we get a constant curvature metric \( g_{S^3/\Gamma(1)}\) on \(S^3/\Gamma\). The push-forward of the euclidean metric on \(\mathbb{C}^2\) by \(\Gamma\) is a flat K\"ahler cone metric on \(S\), isometric to \( dr^2 + r^2  g_{S^3/\Gamma(1)} \). In the search of flat metrics with cone singularities on \( \mathbb{C}^2 \), it is natural to ask whether it is possible to extend \( \Gamma \) to a finite group \( G \subset U(2) \) so that \( \Gamma \subset G \) is normal and the quotient \( H= G/ \Gamma \) acts on \(S\) in a way that \( S/ H \cong \mathbb{C}^2 \) -note that \( S/ H = \mathbb{C}^2 / G \)-. For example, we can look for \(G\) such that \( H \cong \mathbb{Z}_2 \) acts on \(S\) as \( (z, w, t) \to (z, w, -t) \) so that we can push-forward the euclidean metric to get a metric with cone angle \( \pi \) along the plane curve \( C \); as we shall explain this is always possible. Fortunately, finite groups of unitary matrices with the property that \( \mathbb{C}^2/ G \cong \mathbb{C}^2 \) are well understood; these are called unitary reflection groups, we refer to \cite{URG} for the results regarding their classification. 

A unitary linear map $A$ of $\mathbb{C}^n$ is called a reflection if $A$ fixes a hyperplane and $A^m$ is the identity for some $m \geq 2$; equivalently there is an orthonormal basis of $\mathbb{C}^n$ with respect to which $A$ is represented as a diagonal matrix $diag( \epsilon, 1, \ldots, 1)$ with $\epsilon^m =1$. The smallest $m$ is called the order of $A$. A finite group of unitary matrices $G \subset U(n)$ is called a unitary reflection group if it is generated by reflections. A classical theorem of Shephard-Todd-Chevalley characterizes unitary reflection groups as the only finite groups $G$ of unitary matrices with the property that the orbit space  $\mathbb{C}^n /G$ is isomorphic to $\mathbb{C}^n$; or equivalently  the algebra of invariant polynomials $\mathbb{C}[x_1, \ldots, x_n]^G$ is isomorphic to $\mathbb{C}[X_1, \ldots, X_n]$. Shephard-Todd classified these groups. 

Given a unitary reflection group $G \subset U(2)$ let $X_1, X_2 \in \mathbb{C}[x_1, x_2]$ be homogeneous polynomials of smallest degree invariant under the action of $G$ and such that the map $\Phi : \mathbb{C}^2 \to \mathbb{C}^2$ defined as $\Phi = (X_1, X_2)$, factors through the quotient to give an isomorphism $\mathbb{C}^2/G \cong \mathbb{C}^2$. Let $F=\cup_{i=1}^r F_i$ be the union of all complex lines $F_i$ through the origin which are fixed by  some reflection in $G$, this set $F$ coincides with the set of critical points of $\Phi$. We can then push-forward the euclidean metric with $\Phi$ to get a flat K\"ahler metric $\Phi_{*}g_{euc}$ in $\mathbb{C}^2$ with cone singularities along $\Phi(F)$. This metric has cone angle $2\pi \beta_i$ along $\Phi (F_i)$ where $\beta_i = 1/m_i$, with $m_i$ being the least common multiple of the orders of the reflections which fix $F_i$. Since the group $G$ preserves the distance of the points to the origin, it is clear that $\Phi_{*}g_{euc}$ is a Riemannian cone with its apex at the origin, its link is a spherical metric on the three-sphere with cone angle $2\pi \beta_i$ along the intersection of $\Phi(F_i)$ with the unit sphere. Since the $S^1$-action $e^{it}(x_1, x_2) = (e^{it}x_1, e^{it}x_2)$ commutes with the action of $G$ there is an induced $S^1$ action on $\mathbb{C}^2/G$ under which $\Phi_{*} g_{euc}$ is invariant. Indeed this $S^1$-action can be identified with the action generated by the Reeb vector field of $\Phi_{*}g_{euc}$ and it follows that $\Phi_{*}g_{euc}$ must be given by either Proposition \ref{Hfm} or Proposition \ref{Sfm} and therefore it must correspond to a spherical metric in the projective line. 

Before diving into the classification of reflection groups, we analyze the case of \( \Gamma= C_m \). Write \( \omega_m =e^{2\pi i/m} \), so that 

\[ C_m = \langle
\left( \begin{array}{cc}
\omega_{m} & 0 \\
0 & \omega^{-1}_{m} \end{array} \right) \rangle \subset SU(2) .
\]
The invariant polynomials $w=(1/2) (x_1^m + x_2^m)$, $t=(1/2i)(x_1^m -x_2^m)$ and $z=x_1 x_2$ give us the complex isomorphism
$$ \mathbb{C}^2/C_m \cong \{ (z, w, t) \in \mathbb{C}^3 : \hspace{2mm} w^2 + t^2 = z^m \}.$$
Consider the transposition $T(x_1, x_2) = (x_2, x_1)$ and let $G(m, m, 2)$ -notation to be explained later- be the group generated by $C_m$ and $T$, so that $C_m \subset G(m, m, 2)$ is a normal subgroup of index two. The action of $T$ on $(z, w, t) \in \mathbb{C}^2 /C_m$ sends $(z, w, t) \to (z, w, -t)$. We conclude that $\Phi (x_1, x_2) = (z, w)$ is invariant under the action of $G(m, m, 2)$ and gives us a complex isomorphism $\mathbb{C}^2 / G(m, m, 2) \cong \mathbb{C}^2$ the metric $\Phi_{*}g_{euc}$ has cone angle $\pi$ along the curve $\{ w^2= z^m \}$. 

We go further and consider the group $G(2m, 2, 2)$ given by 
\[ G(m, m, 2)= \langle \left( \begin{array}{cc}
\omega_{m} & 0 \\
0 & \omega^{-1}_{m} \end{array} \right), \left( \begin{array}{cc}
0 & 1 \\
1 & 0 \end{array} \right) \rangle  \subset G(2m, 2, 2) = \langle G(m, m, 2), \left( \begin{array}{cc}
\omega_{2m} & 0 \\
0 & \omega_{2m} \end{array} \right) \rangle .\] 
So that \( G(m, m, 2) \subset G(2m, 2, 2) \) is a normal subgroup of index $2m$. The quotient $G(2m, 2, 2)/ G(m, m, 2)$ is cyclic and its generator acts on $(z, w)$  by sending it to $(\omega_m z, -w)$. We conclude that $u=z^m$ and $v=w^2$ are invariant under the action of $G(2m, 2, 2)$ and $\Psi (x_1, x_2) = (u, v)$ gives a complex isomorphism between the orbit space and $\mathbb{C}^2$. Note that $\Psi = S_{(2, m)} \circ \Phi$, where $S_{(2, m)}(z, w) = (z^m, w^2)$. The metric $\Psi_{*} g_{euc}$ has cone angle $\pi$ along the complex lines $\{v=0\}$ and $\{u=v\}$ and cone angle $2\pi (1/m)$ along $\{u=0\}$. We will now see that this correspond under Proposition \ref{Hfm} to a spherical metric $g$ on $\mathbb{CP}^1$ with cone angle $\pi$ at $1$ and $\infty$ and cone angle $2\pi(1/m)$ at $0$. Indeed the components of the map $\Psi (x_1, x_2) = (x_1^m x_2^m, (1/4)(x_1^m + x_2^m)^2)$ are homogeneous polynomials of degree $2m$ and therefore induce a map of the projective line to itself of degree $2m$; in the complex coordinate $\eta = x_1/x_2$ this map writes 
$$\Psi (\eta) = \frac{4 \eta^m}{(1 + \eta^m)^2} .$$
We have that $\Psi(1)=1$, $\Psi (\omega_{2m}) = \infty$ and $\Psi (0) = 0 $; $1$ and $\omega_{2m}$ are critical points of $\Psi$ of order $1$ and $0$ is a critical point of order $m-1$. Let $T$ be the spherical triangle delimited by the arc of the unit circle between $1$ and $\omega_{2m}$ and the two segments of length $1$ connecting $\omega_{2m}$ and $1$ to $0$. We recognize $\Psi$ as a Riemann mapping of $T$ and the spherical metric $g$ on $\mathbb{CP}^1$ as the doubling of $T$. The potential for the euclidean metric is $|x_1|^2 + |x_2|^2$, expressing this in terms of $u$ and $v$ gives the potential for the metric $g_F = \Psi_{*}g_{euc}$ ($r^2$ in Proposition \ref{Hfm} ), up to a constant factor it is
$$\left( h + (h^2 - |u|^2)^{1/2} \right)^{1/m} + \left( h - (h^2-|u|^2)^{1/2} \right)^{1/m}$$
where $h = |v| + |u-v|$. We can use equations \ref{form1} and \ref{form2} to obtain the corresponding expression for the spherical metric $g$ and check that indeed the pull-back of $g$ by $\Psi$ agrees with the standard round metric.  When $m=2$ the expressions simplify to give

$$ r^2 = a \left( |u| + |v| + |u-v| \right)^{1/2}$$
where $a= 8\sqrt{2}$ is determined by the volume normalization condition;  and (using \ref{form1}, \ref{form2})
$$ g = \frac{1}{8} \frac{1}{|\xi| |\xi -1| + | \xi|^2 |\xi-1| + |\xi||\xi-1|^2} |d\xi|^2 . $$
If we write $\xi = \Psi (\eta) = (1+ \eta^2)^2 / (4\eta^2)$, then $\Psi^{*}g= (1+ |\eta|^2)^{-2} |d\eta|^2$ the standard round metric of curvature $4$.

The unitary reflection groups which act irreducibly on $\mathbb{C}^2$ divide into two types: primitive and imprimitive. The group $G$ is called imprimitive if we can find a direct sum decomposition $\mathbb{C}^2 = \mathbb{C}v_1 \oplus \mathbb{C}v_2$ such that the action of $G$ permutes the subspaces $\mathbb{C}v_1$ and $\mathbb{C}v_2$, otherwise it is called primitive. The subspaces $\mathbb{C}v_1, \mathbb{C}v_2$ are said to be a system of imprimitivity for $G$.

\begin{itemize}
	
	\item Let $m>1$ be a natural number and set $\omega_m =e^{2\pi i/m}$. Write $\mathcal{C}_m$ for the cyclic group of $m$-roots of unity generated by $\omega_m$. Let $p \geq 1$ be a natural number that divides $m$ and set $H$ to be the subgroup of  the direct product $\mathcal{C}_m \times \mathcal{C}_m$ consisting of all pairs $(\omega_m^i, \omega_m^j)$ such that $(\omega_m^i \omega_m^j)^{m/p}=1$; note that if $p=1$ then $H= \mathcal{C}_m \times \mathcal{C}_m$. We embed $H$ in $U(2)$ by means of the diagonal action $(\omega_m^i, \omega_m^j) (x_1, x_2) = (\omega_m^i x_1, \omega_m^j x_2)$ and define $G(m, p, 2)$ to be the subgroup of $U(2)$ generated by $H$ and the transposition $T(x_1, x_2) = (x_2, x_1)$. The lines $\mathbb{C}e_1$ and $\mathbb{C}e_2$ form a system of imprimitivity for this group. The notation $G(m, p, 2)$ is due to Shephard-Todd, the $2$ at the end simply means that we are working in two complex dimensions. The fact is that, for $(m, p) \not= (2,2)$, the group $G(m, p, 2)$ is a unitary imprimitive reflection group which acts irreducibly in $\mathbb{C}^2$ and that any such a group is conjugate to some $G(m, p, 2)$ for some values of $m$ and $p$. $G(m, p, 2)$ is a normal subgroup of $G(m, 1, 2)$ of index $p$, the order of $G(m, p, 2)$ is $2m^2/p$. There is a natural inclusion $G(m, 1, 2) \subset G(2m, 2, 2)$ as an index two subgroup induced from $\mathcal{C}_m \subset \mathcal{C}_{2m}$. 
	
	We have already discussed in detail the case of $G(2m, 2, 2)$, the corresponding quotient metric $g_F$ has cone singularities along three complex lines of angles $\pi$, $\pi$ and $2\pi (1/m)$. The general case of the group $G(m, p, 2)$ follows by pulling-back $g_F$ with $(z, w) \to (z^2, w^p)=(u, v)$ thus fitting in with Proposition \ref{Sfm}.
	
\end{itemize}

Let $G$ be a primitive unitary reflection group of $U(2)$ for $g \in G$ take $\lambda_g \in \mathbb{C}$ such that $\lambda_g^2 = \det (g)$. Define $\hat{G} = \{ \pm \lambda^{-1}_g g : \hspace{2mm} g \in G \} \subset SU(2)$. Since $G$ is primitive it follows that $\hat{G}$ must be a binary tetrahedral $T$, octahedral $O$ or icosahedral $I$ group, this splits the primitive subgroups into three types. Shephard-Todd classified these primitive groups by looking at the algebra of invariant polynomials and using the work of Klein on invariant theory for finite subgroups of $SU(2)$, in total there are $19$ primitive groups in $U(2)$ up to conjugation. Denote by $\mathcal{C}_m$ the cyclic group of order $m$ of scalar matrices. Define $\mathcal{T} = \mathcal{C}_{12} \circ T$, $\mathcal{O} = \mathcal{C}_{24} \circ O$ and $ \mathcal{I} = \mathcal{C}_{60} \circ I$; the circle in the notation means the subgroup generated in $U(2)$ by the respective cyclic and binary groups. It turns out that these are primitive unitary groups and that $\hat{G} \subset T$ ($O$ or $I$)  if and only if $ G \subset \mathcal{T}$ ($\mathcal{O}$ or $\mathcal{I}$).

\begin{itemize}
	
	\item  There are $4$ groups of tetrahedral type, all of them are subgroups of $\mathcal{T}$. The order of $\mathcal{T}$ is $12^2=144$. The invariant polynomials can be taken to be $f^3$ and $t^2$ where $f= x_1^4 + 2i \sqrt{3} x_1^2 x_2^2 +x_2^4$ and $t= x_1^5 x_2 - x_1 x_2^5 $. The map $\Phi = (f^3, t^2)$ induces a map in the projective line $\Phi (\eta) = H(t)$, where $\eta = x_1/x_2$, $t=\eta^2$ and
	$$ H(t) = \frac{(t^2 + 2i \sqrt{3} t +1)^3}{t (t^2-1)^2}.$$
	The map $\Phi$ has degree $12$ and looking at its critical points it can be seen to be a Riemann mapping for a spherical triangle with angles $\pi/2$, $\pi/3$ and $\pi/3$. The quotient of the euclidean metric by the group $\mathcal{T}$ is then identified with the lift $g_F$ of the spherical metric with cone angles $\pi$, $2\pi /3$ and $2\pi /3$ by means of the Hopf bundle given by Proposition \ref{Hfm}. The quotient of the euclidean metric by the remaining $3$ tetrahedral groups are obtained as pull-backs of $g_F$ by means of suitable branched covers as in Proposition \ref{Sfm}.

	\item  There are $8$ groups of octahedral type, all of them are subgroups of $\mathcal{O}$. The order of $\mathcal{O}$ is $24^2=576$. The invariant polynomials can be taken to be $h^3$ and $t^2$ where $h= x_1^8 + 14 x_1^4 x_2^4 +x_2^8$ and $t= x_1^{12} - 33x_1^8 x_2^4 - 33x_1^4x_2^8 + x_2^{12} $. The map $\Phi = (h^3, t^2)$ induces a map in the projective line $\Phi (\eta) = H(t)$, where $\eta = x_1/x_2$, $t=\eta^4$ and
	$$ H(t) = \frac{(t^2 + 14 t +1)^3}{(t ^3 - 33t^2 -33t +1)^2}.$$
	The map $\Phi$ has degree $24$ and looking at its critical points it can be seen to be a Riemann mapping for a spherical triangle with angles $\pi/2$, $\pi/3$ and $\pi/4$. The quotient of the euclidean metric by the group $\mathcal{O}$ is then identified with the lift $g_F$ of the spherical metric with cone angles $\pi$, $2\pi /3$ and $2\pi /3$ by means of the Hopf bundle given by Proposition \ref{Hfm}. The quotient of the euclidean metric by the remaining $7$ octahedral groups are obtained as pull-backs of $g_F$ by means of suitable branched covers as in Proposition \ref{Sfm}.
	
	\item There are $7$ groups of icosahedral type, all of them are subgroups of $\mathcal{I}$. The order of $\mathcal{I}$ is $60^2=3600$. The invariant polynomials can be taken to be homogeneous polynomials of degree $60$ which define a map $\Phi $ in the projective line $\Phi (\eta) = H(t)$, where $\eta = x_1/x_2$, $t=\eta^5$ and
	$$ H(t) = \frac{  (t ^4 -228 t^3 + 494 t^2 + 228 t + 1)^3 }{ (t^6 + 522 t^5 - 10005 t^4 - 10005t^2 - 522 t +1)^2}.$$
	The map $\Phi$ has degree $60$ and looking at its critical points it can be seen to be a Riemann mapping for a spherical triangle with angles $\pi/2$, $\pi/3$ and $\pi/5$. The quotient of the euclidean metric by the group $\mathcal{T}$ is then identified with the lift $g_F$ of the spherical metric with cone angles $\pi$, $2\pi /3$ and $2\pi /5$ by means of the Hopf bundle given by Proposition \ref{Hfm}. The quotient of the euclidean metric by the remaining $6$ icosahedral groups are obtained as pull-backs of $g_F$ by means of suitable branched covers as in Proposition \ref{Sfm}.
	
\end{itemize}

Recall that a Riemann mapping from the upper half-plane to a  triangle $T$ whose sides are circle arcs is obtained as the quotient of two linearly independent solutions of the hypergeometric equation, with a suitable value of the parameters $a, b, c$ given in terms of the angles of $T$ . For some particular rational values of the parameters $a, b, c$ the hypergeometric equation has finite monodromy and its solutions are rational functions, these special values are tabulated into the so-called Schwarz's list. In this context; the rational biholomorphisms we have given from the spherical triangles with angles $(\pi /2, \pi /2, \pi / m)$,  $(\pi /2, \pi /3, \pi / 3)$,  $(\pi /2, \pi /3, \pi / 4)$ and  $(\pi /2, \pi /3, \pi / 5)$, correspond precisely with the cases in Schwarz's list in which the associated triangle is spherical and its angles are integer quotients of $\pi$.

\section{Limits of K\"ahler-Einsten metrics with cone singularities} \label{CURVES}

\subsection{Singular points of plane complex curves}

Let $C=\{ f = 0 \} \subset \mathbb{C}^2$ be a complex curve with an isolated singularuty at $0$ and let $B$ be a small ball around the origin. Fix $ 0 < \beta < 1$;  we want to discuss possible notions of a K\"ahler metric $g$ on $B$ with cone angle $2 \pi \beta$ along $C$. Outside the origin there are standard definitions so that the key point is to say what is the behavior of $g$ at $0$. First thing to say is that we  want the volume form of $g$ to be locally integrable, moreover we also require that
$$ \mbox{Vol} (g) = G |f|^{2\beta -2} \Omega \wedge \overline{\Omega} $$
where  $G$ is a \emph{continuos} function and $\Omega = dzdw$ is the standard holomorphic volume form. This leads to the so-called complex singularity exponent of the curve $c_0 (f)$. The number $c_0 (f)$ is defined as the supremum of all $c>0$ such that $|f|^{-2c} \Omega \wedge \overline{\Omega}$ is locally integrable. This is always a rational number and can be computed in algebro-geometric terms by means of successive blow-ups of the singularity. It is clear that $0 < c_0 \leq 1$ and indeed $c_0 (f) =1$ only when the curve is smooth or has a simple double point at $0$. Write $f = P_d + (h.o.t.)$ with $P_d$ a homogeneous polynomial of degree $d$ and $(h.o.t.)$ meaning higher order terms; according to \cite{kollar}
\begin{equation}
	c_0 (f) = \frac{1}{d} + \frac{1}{e} , 
\end{equation}
where $e/d$ is the first Puiseux exponent of $f$ (see \cite{brieskorn}). In terms of cone angles we must take  
\begin{equation} \label{csec}
	\beta > 1 -c_0 .
\end{equation}
Indeed, if \ref{csec} holds then there is the notion of a \emph{weak K\"ahler metric} (see \cite{EGZ}). Pluri-potential theory provides, for any $ 1- c_0 < \beta < 1$, a weak K\"ahler metric $g$ with  $ \mbox{Vol}(g) =|f|^{2\beta -2} \Omega \wedge \overline{\Omega}$. A draw-back of this approach is that little can be said on the geometry of the metric, in particular there is no guarantee of the existence of a tangent cone at $0$. On the other hand if the metric $g$ arises as the Gromov-Hausdorff limit of a sequence of K\"ahler metrics with cone singularities along smooth curves then, under a suitable assumptions, it will have a tangent cone at $0$. We discuss plausible, stronger notions according to the type of singularity of $C$.

\begin{itemize}
	\item \emph{Ordinary multiple points.} In this case the zero set of $P_d$ consists of $d$ distinct complex lines and $c_0(f)=2/d$, so that $1-c_0 = (d-2)/d$. On the other hand the Troyanov condition \ref{Tcondition} is equivalent to
	$$ \frac{d-2}{d} < \beta < 1 $$
	when all the angles $\beta_i$ are equal to $\beta$. Therefore, for $\beta$ in this range we  have the flat metric $g_F$ given by Proposition \ref{Hfm}. Let us assume first that there  are suitable holomorphic coordinates around $0$ in which $C = \{ P_d =0 \}$; then we can require the condition that $\| g - g_F \|_{g_F} = O (r^{\epsilon})$ for some $\epsilon >0$ as $r \to 0$. Indeed if this condition holds in a little bit stronger H\"older sense, then it is straightforward to show that $g$ has a tangent cone at $0$ which agrees with $g_F$. In the general case such a holomorphic change of coordinates doesn't exists, but we can use a diffeomorphism $\Phi$ of the ball,  sufficiently close to the identity, which takes $C$ to the zero set of $P_d$. Indeed, in a small ball around the origin, the curve $C$ consists of $d$ branches, each of which is the graph of a holomophic function over one of the lines of $\{ P_d =0 \}$. It is then not hard to construct $\Phi$ by means of suitable cut-off functions, moreover $\Phi$ can be taken to be holomorphic in a suitable neighborhood of the curve.
	
	\item \emph{Non-ordinary multiple points.} Let us consider first, as a model example, the case of the cusp 
	
	$$C = \{ w^2 = z^3 \} .$$
	In this case the complex singularity exponent is equal to $c_0 =5/6$, so that $1-c_0 =1/6$. We have seen that for any $ 1/6 < \beta < 5/6$ we have the flat cone metric $\tilde{g_F}$ given by Proposition \ref{Sfm} with $\mbox{Vol}(\tilde{g}_F) = |w^2 - z^3|^{2\beta -2} \Omega \wedge \overline{\Omega}$. The same as before we can consider metrics $g$ which satisfy the condition $\| g - \tilde{g}_F \|_{\tilde{g}_F} = O (r^{\epsilon})$ for some $\epsilon >0$ as $r \to 0$; and if this condition holds in a  H\"older sense, then $g$  has a unique tangent cone at $0$ which agrees with $\tilde{g}_F$. The question is what to do when $ 5/6 \leq \beta < 1$. We look back at the picture, Figure \ref{Pacman},  of two cone angles of total angle $2\pi \beta$ that collide and produce a cone angle $2\pi \gamma$ with $\gamma = 2\beta -1$. We expect  that, in transverse directions to $C$, we should see two cone singularities coming together.  As a simple model consider the metric
	\begin{equation}
		g= |dz|^2 + |w^2 - z^3|^{2\beta-2} |dw|^2 .
	\end{equation}
	This is not a K\"ahler metric, but it has the right volume form. If we fix $z_0$ and embed $\mathbb{C}$ into $\mathbb{C}^2$ by means of $ \tau_{z_0} (w) = (z_0, w)$; then 
	$$ \tau_{z_0}^{*}g= |w-a|^{2\beta -2} |w+a|^{2\beta-2} |dw|^2, $$
	where $a^2=z_0^3$. This is a flat metric in $\mathbb{C}$ with two cone singularities of angle $2\pi\beta$ at $a$ and $-a$. If we let $z_0 \to 0$ then $a \to 0$ and $\tau_{0}^{*} g = |w|^{2\gamma -2} |dw|^2$ with $\gamma = 2\beta -1$. We shall see now that if $\beta >5/6$ then the tangent cone at $0$ of $g$ is the metric $g_{(\gamma)} = |dz|^2 + |dw|^{2\gamma -2} |dw|^2$. Let $D_{\lambda} (z,w)=(\lambda z, \lambda^{1/\gamma} w)$, so that $D_{\lambda} g_{(\gamma)} = \lambda^2 g_{(\gamma)}$. It requires a simple computation to check that
	$$ \lambda^{-2} D_{\lambda}^{*} g = |dz|^2 + |w^2 - \lambda^{3-2/\gamma} z^3|^{2\beta -2} |dw|^2 .$$
	We see that  $\lambda^{-2} D_{\lambda}^{*} g$ converges to $g_{(\gamma)}$ as $\lambda \to 0$ provided that $ 3- 2/\gamma >0$, which is the same to say $\beta >5/6$.
	
	The same discussion applies to the more general case of the curve $C=\{ w^m = z^n \}$ with $2 \leq m < n$. So that $c_0 = 1/m + 1/n$ and we have seen that for $1-1/m-1/n < \beta < 1-1/m +1/n$ there is a flat cone metric $\tilde{g}_F$ with volume form $|w^m -z^n|^{2\beta -2} \Omega \wedge \overline{\Omega}$. This time the tangent cone at $0$ should be given, when $ 1-1/m+1/n< \beta$,  by $g_{(\gamma)}$ with $1-\gamma = m (1- \beta)$.
	
	We have little to say in the general case of an arbitrary singularity. If $m$ is the order and $n/m$ is the first non-zero Puiseux exponent; then the intersection of the curve with small spheres around the origin is an iterated torus knot over the $(m, n)$-torus knot -see \cite{milnor}- and there is no homeomorphism which takes the curve to the singular set of a flat cone metric. Naively, we might expect that the tangent cone is the metric $\tilde{g}_F$ given by Proposition \ref{Sfm} with cone angle $\beta$ along $\{ w^m= z^n \}$ when $1-1/m-1/n < \beta < 1-1/m +1/n$; and it is $g_{(\gamma)}$ with $1-\gamma = m (1- \beta)$ when $ 1-1/m+1/n< \beta$.
	
\end{itemize}

\subsection{Blow-up analysis}
As we said in the Introduction,  the $L^2$-norm of the Riemannian curvature tensor of a K\"ahler-Einstein metric $g$ on a complex surface $X$ with cone angle $2\pi\beta$ along  a smooth curve $D$ is given in terms of topological data by means of the formula
\begin{equation} \label{comen}
	E(g) = \frac{1}{8\pi^2} \int |\mbox{Riem}|^2 = \chi (X) + (\beta-1) \chi (D) ,
\end{equation}
where $\chi$ denotes the Euler characteristic.
The number $E(g)$ is the so-called `energy' of $g$. In four real dimensions the energy is a scale-invariant quantity, i.e. $E(g) = E(\lambda g)$ for any $\lambda >0$. This fits the theory into a blow-up analysis framework which parallels the case of smooth Einstein metrics on four manifolds \cite{anderson}: The only possible way in which a non-collapsing sequence of solutions $g_i$ can degenerate is when the energy distributions  $|\mbox{Riem}(g_i)|^2$ develop Dirac deltas, this can happen only at finitely many points. Let $p$ be such a point. Re-scaling the metrics $g_i$ at $p$ in order to keep the Riemannian curvature bounded one gets in the limit a Ricci-flat metric with finite energy  on a non-compact space,  so-called `ALE gravitational  instanton' or, more generically, a `bubble'. There might be many different (but finite) blow-up limits at $p$; and these can be arranged into a `bubble tree' associated to $p$, the tangent cone at infinity of the `deepest bubble' in the tree agrees with the tangent cone of the (non-scaled) limiting solution at $p$. If we add the energy of the singular limit space with the energy of the blow-up limits we recover the energy of the original sequence $g_i$.

We are interested in the case of a non-collapsed sequence of K\"ahler-Einstein metrics $g_i$ on a complex surfaces $X_i$ with cone angle $2\pi\beta$ along smooth curves $D_i$ within some fixed numerical data. There is a weak K\"ahler-Einstein metric on the Gromov-Hausdorff limit $W$ with cone singularities along a Weil divisor $\Delta$. There is a decomposition $\Delta = \cup_{k=1}^M \Delta_k$ where $\Delta_k$ is the component of $\Delta$ of multiplicity $k$. For the sake of definiteness we consider the case when $p$ is a singular ordinary multiple-point of $\Delta_1$ which lies on the smooth part of $W$. The tangent cone of $W$ at $p$ must be given by Proposition \ref{Hfm} with $\beta_j = \beta$ for $j=1, \ldots d.$ The blow-up of the metrics $g_i$ at $p$ results in a Ricci-flat metric on $\mathbb{C}^2$ with cone angle $2\pi\beta$ along a complex curve of degree $d$ with $d$ distinct asymptotic lines and $g_F$ as its tangent cone at infinity. These metric were shown to exist in \cite{martin}. Yau's work on the Calabi conjecture has been extended to the setting of ALE and asymptotically conical manifolds -\cite{Joyce}, \cite{ConlonHein}-, and to the context of metrics with cone singularities -\cite{brendle}, \cite{JMR}-; the proof of the existence theorem in \cite{martin} is a mix of these articles.

\begin{remark} 
These blow-up limits of K\"ahler-Einstein metrics with cone singularities arose first in the context of the `deformation of the cone angle method' used to establish the existence of K\"ahler-Einstein metrics on K-stable Fano manifolds, see \cite{donaldson1}. Let $X$ be a Fano manifold and $D \subset X$ a smooth anti-canonical divisor, it is known that for small values of $\beta$ there is a KE metric on $X$ with cone angle $2 \pi \beta$ along $D$. The question is to understand the behavior of these metrics as $\beta$ increases in manifolds which are not $K$-stable. We refer to \cite{donaldson1} and \cite{szek} for a discussion when $X$ is the complex projective plane blown-up at one and two points. 
\end{remark}

There is a well-known formula for the energy of an ALE gravitational instanton. If $\Gamma $ is a finite subgroup of $SU(2)$ acting freely on $S^3$ and $g$ is a Ricci-flat metric on $M$ asymptotic to the cone over $S^3/ \Gamma$, then
\begin{equation} \label{alen}
	E(g) = \chi (M) - \frac{1}{|\Gamma|} ,
\end{equation} 
where $|\Gamma|$ is the order of $\Gamma$. Note that $1/|\Gamma|$ is the volume ratio $\mbox{Vol}(S^3/ \Gamma )/ \mbox{Vol}(S^3)$. Let now $g_{RF}$ be a Ricci-flat K\"ahler metric on $\mathbb{C}^2$ with cone angle $2\pi\beta$ along the smooth complex curve $C \subset \mathbb{C}^2$. Let the tangent cone at infinity be the cone over the spherical metric  on the 3-sphere $\overline{g}$ with cone singularities. Under suitable assumptions on the regularity and asymptotic behavior of $g_{RF}$ it is reasonable to expect that the energy is given by a formula which mixes \ref{comen} and \ref{alen} (see \cite{martin})
\begin{equation} \label{ENERGY}
	E (g_{RF}) = 1 + (\beta -1) \chi (C) - \frac{\mbox{Vol}(\overline{g})}{2\pi^2} .
\end{equation}
The number $1$ is the Euler characteristic of a ball. As we mentioned before, the volume ratio $\nu = \mbox{Vol}(\overline{g}) / 2 \pi^2$ measures how bad the singularity of the limit $W$ is at $p$. The energy of the bubbles at $p$ are bigger as $\nu$ is smaller. 

On the other hand; a straight-forward application of the Bishop-Gromov volume monotonicity formula gives us a lower bound on the volume ratio, see \cite{OSS}. If we  assume that the curves $D_i$ all lie in the linear system $H^0(L)$, then
\begin{equation} \label{bishgro}
	\nu \geq \frac{1}{9} (c_1(X) - (1-\beta)c_1(L))^2 .
\end{equation}
This inequality can be used to rule out, for example, the degeneration of KE metrics in $\mathbb{CP}^2$ with cone angle $2\pi\beta$ along smooth cubics to a cubic with a cuspidal point $\{w^2 = z^3 \}$. Indeed at such a singular point the tangent cone should be $\mathbb{C}_{\gamma} \times \mathbb{C}$ with $\gamma = 2\beta -1$ when $\beta > 5/6$, so that $\nu = 2\beta -1$. Replacing this into \ref{bishgro} we get   
\begin{equation*}
2\beta -1 \geq \beta^2 , 
\end{equation*}
which holds only when $\beta=1$.

\emph{Line Arrangements.}
Consider  a collection of lines $L_1, \ldots, L_k$  in $\mathbb{CP}^2$.  An $r$-tuple point is a point where $r$ lines of the arrangement meet, we denote by $t_r$ the number of $r$-tuple points. Since any two lines meet at exactly one point we get the identity

\begin{equation*}
\frac{k(k-1)}{2} = \sum_{r \geq 2} t_r \frac{r (r-1)}{2} .
\end{equation*}

The arrangement is said to have the Hirzebruch property if $k=3n$, $n\geq 2$ and each line intersects the others at exactly $n+1$ points. Such arrangements where considered by Hirzebruch \cite{Hir}, in a construction of compact quotients of the unit ball as ramified covers of the projective plane.  It was shown by Panov in \cite{panov} that there is a polyhedral K\"ahler metric with cone angle $\beta = \frac{n-1}{n}$ along any of these arrangements, which is unique up to scale. The only known examples of arrangements which satisfy the Hirzebruch property so far are associated with unitary reflection groups of $U(3)$, there are two infinite families and five exceptional cases. The simplest example of these arrangements, $A_0(2)$, consists of the extended sides of a triangle together with its three bisectrices.

Let $P_0$ be a homogeneous polynomial of degree $k=3n$ such that $\{P_0 =0\} \subset \mathbb{CP}^2$ is a line arrangement $L_1, \ldots, L_k$ with the Hirzebruch property and let $g_0$ be the polyhedral K\"ahler metric with cone angle $\beta = \frac{n-1}{n}$ along the arrangement. Let $C_{\epsilon} = \{P_{\epsilon}=0 \} \subset \mathbb{CP}^2$ for $\epsilon >0$ be a family of smooth curves of degree $k$ which converge to the arrangement as $\epsilon \to 0$ and let $g_{\epsilon}$ be the Ricci-flat metric in the projective plane with cone angle $2\pi\beta$ along the curve $C_{\epsilon}$. One might expect that, under suitable hypothesis, the metrics $g_{\epsilon}$ converge to $g_0$ in the Gromov-Hausdorff sense as $\epsilon \to 0$. Write $E$ for the energy of the metrics $g_{\epsilon}$, which can be computed from \ref{comen} and the degree-genus formula. Let $E_r$ be the energy of a Ricci-flat metric on $\mathbb{C}^2$ with cone angle $2\pi \beta$ along a smooth curve of degree $r$ with $r$ different asymptotic lines given by equation \ref{ENERGY}. Since
the metric $g_0$ is flat, and therefore it has $0$ energy, one would expect the identity
\begin{equation} \label{lines energy}
E = \sum_{r} t_r E_r 
\end{equation}
in the absence of bubble tree phenomena. It follows from straightforward computations that equation \ref{lines energy} holds for all the listed arrangements:

\begin{itemize}
	\item The family $A_0 (m)$. 
	\begin{equation*}
	k=3m, \hspace{3mm} t_3= m^2, \hspace{3mm} t_m =3
	\end{equation*}
	
	\begin{equation*}
	\beta=\frac{m-1}{m}, \hspace{3mm} E= 9m-6, 
	\end{equation*}
	
	\begin{equation*}
	E_3= 1 + \frac{3}{m} - \frac{(2m-3)^2}{4m^2}, \hspace{3mm} E_m = 1 + \frac{m^2-2m}{m} -\frac{1}{4}, 
	\end{equation*}
	
	\begin{equation*}
	E= m^2 E_3 + 3 E_m .
	\end{equation*}
	
	\item The family $A_3 (m)$, $m \geq 2$. 
	\begin{equation*}
	k=3m +3, \hspace{3mm} t_2= 3m, \hspace{3mm} t_3 =m^2, \hspace{3mm} t_{m+2}=3
	\end{equation*}
		
	\begin{equation*}
	\beta=\frac{m}{m+1}, \hspace{3mm} E= 3 + \frac{(3m+2)(3m+1)-2}{m+1}, 
	\end{equation*}
	
	\begin{equation*}
    E_2 = 1 - \frac{m^2}{(m+1)^2}, \hspace{3mm} E_3= 1 + \frac{3}{m+1} - \frac{(2m-1)^2}{4(m+1)^2}, \hspace{3mm} E_{m+2} = 1 + \frac{(m+2)^2-2m -4}{m+1} -\frac{m^2}{4(m+1)^2}, 
	\end{equation*}
	
	\begin{equation*}
	E= 3m E_2 + m^2 E_3 + 3 E_{m+2}.
	\end{equation*}

	\item The Hesse arrangement. 
	\begin{equation*}
	k=12, \hspace{3mm} t_2= 12, \hspace{3mm} t_4 =9
	\end{equation*}

	\begin{equation*}
	\beta=\frac{3}{4}, \hspace{3mm} E= 3 + \frac{11\cdot 10 -2}{4}=30, 
	\end{equation*}
	
	\begin{equation*}
	E_2= 1 - \frac{9}{16}, \hspace{3mm} E_4 = 1 + \frac{16-8}{4} -\frac{1}{4}, 
	\end{equation*}
	
	\begin{equation*}
	E= 12 E_2 + 9 E_4 .
	\end{equation*}

    \item The extended Hesse arrangement. 
	\begin{equation*}
	k=21, \hspace{3mm} t_2= 36, \hspace{3mm} t_4 =9, \hspace{3mm} t_5=12
	\end{equation*}

	\begin{equation*}
	\beta=\frac{6}{7}, \hspace{3mm} E= 57, 
	\end{equation*}
	
	\begin{equation*}
	E_2 = 1 - \frac{6^2}{7^2}, \hspace{3mm} E_4 = 1 + \frac{8}{7} -\frac{10^2}{4 \cdot 7^2}, \hspace{3mm} E_5 = 1+ \frac{15}{7} - \frac{9^2}{4 \cdot 7^2}
	\end{equation*}
	
	\begin{equation*}
	E = 36 E_2 + 9 E_4 + 12 E_5 .
	\end{equation*}

	\item The icosahedral arrangement. 
	\begin{equation*}
	k=15, \hspace{3mm} t_2= 15, \hspace{3mm} t_3 =10, \hspace{3mm} t_5=6
	\end{equation*}

	\begin{equation*}
	\beta=\frac{4}{5}, \hspace{3mm} E= 3+ \frac{14\cdot 13 -2}{5} =39, 
	\end{equation*}
	
	\begin{equation*}
	E_2= 1 - \frac{16}{25}, \hspace{3mm} E_3 = 1 + \frac{3}{5} -\frac{49}{100}, \hspace{3mm} E_5= 1+ \frac{15}{5} - \frac{1}{4}, 
	\end{equation*}
	
	\begin{equation*}
	E= 15 E_2 + 10 E_3 + 6 E_5.
	\end{equation*}

	\item $G_{168}$ arrangement. 
	\begin{equation*}
	k=21, \hspace{3mm} t_3= 28, \hspace{3mm} t_4 =21
	\end{equation*}

	\begin{equation*}
	\beta=\frac{6}{7}, \hspace{3mm} E= 3 + \frac{20 \cdot 19 -2}{7}=57, 
	\end{equation*}
	
	\begin{equation*}
	E_3= 1 + \frac{3}{7} - \frac{11^2}{4 \cdot 7^2}, \hspace{3mm} E_4 = 1 + \frac{8}{7} -\frac{10^2}{4 \cdot 7^2}, 
	\end{equation*}
	
	\begin{equation*}
	\hspace{3mm} E= 28 E_3 + 21 E_4 .
	\end{equation*}

	\item The $A_6$ configuration. 
	\begin{equation*}
	k=45, \hspace{3mm} t_3= 120, \hspace{3mm} t_4 =45, \hspace{3mm} t_5=36
	\end{equation*}

	\begin{equation*}
	\beta=\frac{14}{15}, \hspace{3mm} E= 3+ \frac{44 \cdot 43 -2}{15}= 129, 
	\end{equation*}
	
	\begin{equation*}
	 E_3= 1 + \frac{3}{15} - \frac{27^2}{4 \cdot 15^2}, \hspace{3mm} E_4= 1 + \frac{8}{15} - \frac{26^2}{4 \cdot 15^2}, \hspace{3mm} E_5 = 1 + \frac{15}{15} -\frac{25^2}{4 \cdot 15^2}, 
	\end{equation*}
	
	\begin{equation*}
	E= 120 E_3 + 45 E_4 + 36 E_5.
	\end{equation*}
		
\end{itemize}

Indeed; \ref{lines energy}  should be equivalent to Equation 3 in Theorem 1.11 of \cite{panov}, which expresses the Euler characteristic of the projective plane in terms of the residues of the flat logarithmic connection corresponding to the PK metric.

\subsection{Moduli Spaces}
In general lines we can say that if we fix $\beta \in (0,1)$ and consider pairs $(X, D)$ of a complex manifold together with a smooth divisor which satisfy some fixed numerical data, then the Gromov-Hausdorff compactification $\overline{\mathcal{M}}^{GH}_{\beta}$ of the moduli space of KE metrics on $X$ with cone angle $2\pi\beta$ along $D$ is expected to agree with a suitable algebraic compactification. Of course these compactifications depend on the parameter $\beta$ but, since they agree on the open subset of smooth divisors, all of them are birrationaly equivalent. More precisely, there should be a discrete sequence of cone angles $ 0 < \ldots < \beta_2 < \beta_1 < \beta_0 =1$ such that the spaces  $\overline{\mathcal{M}}^{GH}_{\beta}$ are all isomorphic if $\beta \in (\beta_i, \beta_{i-1})$ and  if $\tilde{\beta} \in (\beta_{i+1}, \beta_i)$ then   $\overline{\mathcal{M}}^{GH}_{\tilde{\beta}}$ is obtained from  $\overline{\mathcal{M}}^{GH}_{\beta}$ by means of a suitable blow-up. The situation can be compared with that of variations of GIT quotients. 

It is shown in \cite{OSS} that the GIT compactiﬁcation of cubic surfaces in \( \mathbb{CP}^3 \) corresponds to the Gromov-Hausdorff compactification of the space of KE metrics on del Pezzo surfaces of degree three. On the other hand; Garcia-Gallardo \cite{garcia} used GIT to construct, for each $ t \in [0, 1] \cap \mathbb{Q}$, a compactification of the moduli space of cubic surfaces together with anticanonical divisors. It is expected that these agree with the Gromov-Hausdorff compactifications of the space of corresponding KEcs metrics with cone angle \( \beta = 1-t \), see \cite{garcia}.

A somewhat different case is that of smooth curves of degree $n \geq 3$ in $\mathbb{CP}^2$. It is known that for $(n-3)/n < \beta < 1$ there is a KEcs, unique up to scale, on the projective plane with cone angle $2\pi\beta$ along a given smooth curve of degree $n$. The set $\mathcal{A}$ of all these curves in the projective plane modulo projective transformations has the structure of an affine algebraic variety and a natural GIT compactification
$$ \overline{\mathcal{A}}_{GIT} = \mathbb{P}(\mbox{Sym}^n\mathbb{C}^3) //SL(3, \mathbb{C}) .$$
It is expected that for $\beta$ sufficiently close to $1$ it holds that $\overline{\mathcal{M}}^{GH}_{\beta} \cong \overline{\mathcal{A}}_{GIT} $. We illustrate these ideas in the particular cases of $n=3$ and $n=4$.

\subsection*{$n=3$: Elliptic curves in $\mathbb{CP}^2$} $\mathcal{A} \cong \mathbb{C}$ and there is only one algebraic compactification  obtained by adding a single point. In the GIT compactification, $ \overline{\mathcal{A}}_{GIT} = \mathbb{P}(\mbox{Sym}^3\mathbb{C}^3) //SL(3, \mathbb{C}) \cong \mathbb{CP}^1$, this extra point is represented by the polystable curve $C_0 = \{ x_0 x_1x_2=0 \} \subset \mathbb{CP}^2$. On the other hand, if we fix  $0<\beta<1$, there is an explicit KE (indeed constant holomorphic sectional curvature) metric $g_0$ with cone angle $2\pi \beta$ along $C_0$ obtained as a K\"ahler quotient of $(\mathbb{C}_{\beta}^*)^3$ by the $S^1$ action $e^{i\theta}(x_1, x_2, x_3)=(e^{i\theta}x_1, e^{i\theta}x_2, e^{i\theta}x_3)$;  if $\beta=1/k$ then $g_0$ is the push forward of the Fubini-Study metric under the map $ [x_0, x_1, x_2] \to [x_0^k, x_1^k, x_2^k]$. The curve $C_0$ has three ordinary double point singularities and the tangent cone of $g_0$ at any of these points is $\mathbb{C}_{\beta} \times \mathbb{C}_{\beta}$. We expect that $\overline{\mathcal{M}}^{GH}_{\beta} = \mathcal{M}_{\beta} \cup \{g_0\} $.
	
	We relate this picture to the blow-up phenomena discussed in the previous section. Let
	
	$$ C_{\epsilon} =  \lbrace x_0 x_1 x_2 - \epsilon (x_0^3 + x_1^3 + x_2^3) = 0 \rbrace .$$
	These are smooth pair-wise non-isomorphic elliptic curves for small $\epsilon>0$. Set $g_{\epsilon}$ to be the corresponding metrics in $\mathcal{M}_{\beta}$. We expect that $g_{\epsilon} \to g_0$ in the Gromov-Hausdorff sense as $\epsilon \to 0$.
	Take coordinates centered at $p=[0,0,1]$  given by $(u, v) \to [u,v,1]$, so that 
	$ C_{\epsilon}= \lbrace uv= \epsilon(u^3 +v^3 + 1) \rbrace$.
	Write $ u = \sqrt{\epsilon} z$ and $ v = \sqrt{\epsilon} w$ so that 
	$$C_{\epsilon}= \lbrace zw= \epsilon^{3/2} z^3 + \epsilon^{3/2} w^3 + 1 \rbrace .$$
	In the coordinates $(z, w)$ the curves $C_{\epsilon}$ converge to $C=\{ zw=1\}$ as $\epsilon \to 0$. By means of the Gibbons-Hawking ansatz, Donaldson \cite{donaldson1} constructed a Ricci-flat K\"ahler metric $g_{RF}$ on $\mathbb{C}^2$  invariant under the $S^1$-action $e^{i\theta}(z, w) = (e^{i\theta} z, e^{-i\theta} w)$ with cone angle $2\pi\beta$ along $C = \lbrace zw =1 \rbrace$ and tangent cone at infinity $\mathbb{C}_{\beta} \times \mathbb{C}_{\beta}$. We expect that  $ (\mathbb{CP}^2, \lambda_{\epsilon} g_{\epsilon}, p) \to (\mathbb{C}^2, g_{RF}, 0)$  in the pointed Gromov-Hausdorff sense, where $\lambda_{\epsilon}$ is a fixed constant multiple of $|\mbox{Riem}(g_{\epsilon})|(p)$. The same discussion applies if $p$ is replaced $[1,0,0]$ or $[0,1,0]$. Donaldson computed the Riemann curvature tensor of $g_{RF}$ and it is not hard from here to compute the energy of the metric to obtain $E(g_{RF})=1-\beta^2$ (which agrees with \ref{ENERGY} since $\chi(C)=0$ and $\mbox{Vol}(\overline{g}) = 2\pi^2 \beta^2$). On the other hand, by \ref{comen}, $E(g_{\epsilon})= \chi(\mathbb{CP}^2) + (\beta-1)\chi(C_{\epsilon})=3$. It is easy to show that $E(g_0)=3\beta^2$. Our speculations are then compatible with the fact that  
	
	$$E(g_{\epsilon}) - E(g_0) = 3- 3\beta^2 = 3 (1-\beta^2)=3E(g_{RF}) .$$
	The limiting metric $g_0$ has less energy than the metrics in the family $g_{\epsilon}$ and the energy lost is due to the formation of three bubbles $g_{RF}$ at the double points of $C_0$.

	\subsection*{$n=4$: Genus $3$ curves in $\mathbb{CP}^2$} The affine variety parameterizing smooth quartic curves modulo projective transformations has complex dimension six,  $\dim_{\mathbb{C}} \mathcal{A}=6$.  The geometric invariant theory for quartic curves is well-understood: The stable points of $\overline{\mathcal{A}}_{GIT}$ parametrize quartic curves with at worst singularities of type $A_1$ or $A_2$; the polystable points  form a $1$-dimensional family which parametrizes either the double conic $Y_1 = \{ Q^2=0\} \subset \mathbb{CP}^2$, where $Q=x_2^2 + x_0 x_1$, or a union of two reduced conics that are tangential at two points and at least one of them is smooth  $Y_{\lambda} = \{P_{\lambda} =0\}$ where $P_{\lambda}= (\lambda_1 x_2^2 + x_0 x_1) (\lambda_2 x_2^2 +  x_0 x_1)$ with $\lambda = [\lambda_1, \lambda_2] \in \mathbb{CP}^1 \setminus \{[1,1]\}$. Note that $[\lambda_1, \lambda_2]$ parametrizes the same curve as $[\lambda_2, \lambda_1]$; when $\lambda$ is $0 = [0,1]$ or $\infty = [1,0]$ the curve $Y_0 =Y_{\infty}$ is referred as the ox, otherwise it is called a cateye (the names are suggested by the graphs of the curves). A cateye has two $A_3$ singularities (also known as tacnodes); the ox has two tacnodes and one $A_1$ singularity. 
	
	Set $\overline{\mathcal{A}}_{1/2}$ to be the blow-up of $\overline{\mathcal{A}}_{GIT}$ at the double conic. If $ E \subset \overline{\mathcal{A}}_{1/2}$ denotes the exceptional divisor, then $E$ is identified with the space of hyperelliptic Riemann surfaces of genus $3$ or equivalently the GIT quotient of the space of eight unordered points in the Riemann sphere by the action of M\"obius transformations
	\begin{equation*}
	E \cong \mathbb{P}(\mbox{Sym}^8\mathbb{C}^2) //SL(2, \mathbb{C}) .
	\end{equation*}
	A point $p \in E$ is represented by a homogeneous degree $8$ polynomial in two variables $p = [f_8]$ and $\{f_8 =0\} \subset \mathbb{CP}^1$ is a configuration of (at most eight) points with multiplicities, $p$ is stable if each point has multiplicity at most three and there is only one polystable point which corresponds to the configuration of two points with multiplicity four. The point $p$ parametrizes the degree $8$  curve $Y_p = \{x_2^2= f_8 (x_0, x_1) \} \subset \mathbb{P}(1,1,4)$. The curve $Y_p$ does not go through the $\frac{1}{4} (1,1)$-orbifold point $[0,0,1] \in \mathbb{P} (1,1,4)$; if $p$ is stable the curve $Y_p$ has at worst singularities of type $A_1$ and $A_2$ (which correspond to points in the configuration of multiplicity $2$ and $3$ respectively), the polystable point parametrizes the curve $ \{ x_2^2 = x_0^4 x_1^4 \} \subset \mathbb{P}(1,1,4) $ with two tacnodal singularities at the points $[1,0,0]$ and $[0,1,0]$.
	 
	There is a classical dichotomy for Riemann surfaces of genus $3$: Either the bi-canonical map defines an embedding of the curve in the projective plane or the canonical map is a degree two map to the projective line branched at eight points (the hyperelliptic case). Therefore $\overline{\mathcal{A}}_{1/2}$ is a compactification of the space of genus $3$ Riemann surfaces.
	
	On the other hand, the anti-canonical map of any degree two Del Pezzo surface $X$ defines a $2$-sheeted covering of the projective plane branched along a smooth quartic curve and vice-versa; the deck transformation $\sigma$ is known as the Geiser involution. There is a  KE metric on $X$ -unique up to scale- which is necessarily invariant under $\sigma$; the push-forward of this metric to the projective plane is a KE metric with cone angle $\pi$ ($\beta = 1/2$) along the  quartic. It is shown in \cite{OSS} that
	
	\begin{equation*}
	\overline{\mathcal{M}}^{GH}_{1/2} \cong \overline{\mathcal{A}}_{1/2} ,
	\end{equation*}
	in the sense that the natural map which sends a K\"ahler metric to its parallel complex structure defines a homeomorphism. The curve singularities which appear at limit spaces are either of type $A_1, A_2$ or $A_3$, and we have discussed in detail their tangent cones -for \( \beta = 1/2 \)- in Subsection \ref{URG}.   
	
	Fix now $\beta$ sufficiently close to $1$ and let $1 - \gamma=2(1-\beta)$. For $\gamma>1/4$ there is a unique KE metric $g_0$ in $2\pi c_1(\mathbb{CP}^2)$ with positive Ricci curvature and cone angle $2\pi\gamma$ along the conic $C_0=\{Q=0 \}$ (see \cite{LiSong}). Let $F$ be a generic polynomial of degree $4$ and let $C_{\epsilon} = \lbrace Q^2 - \epsilon F \rbrace =0 $. Write $Z = \lbrace F =0 \rbrace$, so that for a typical $F$ the intersection $Z \cap C_0$ consists of $8$ distinct points $p_1, \ldots,  p_8$. For small and non-zero $\epsilon$ the curve $C_{\epsilon}$ is smooth; orthogonal projection to $C_0$ is an `approximately' holomorphic double cover from $C_{\epsilon}$ to $C_0$, branched over the points $\{ p_1, \ldots p_8 \}$. Let $g_{\epsilon}$ be the metric on $\mathcal{M}_{\beta}$ corresponding to $C_{\epsilon}$; we expect that $g_{\epsilon} \to g_0$ in the Gromov-Hausdorff sense as $\epsilon \to 0$. The divisor $E$ in $\overline{\mathcal{M}}^{GH}_{1/2}$ should then be contracted to a point in $\overline{\mathcal{M}}^{GH}_{\beta}$. We say a few words on the blow-up limits that might arise in this situation.  Let $C = \lbrace w = z^2 \rbrace \subset \mathbb{C}^2$; there should be a Ricci-flat metric $g_{RF}$ with cone angle $2\pi\beta$ along $C$ asymptotic to the cone $\mathbb{C}_{\gamma} \times \mathbb{C}$. The energy of $g_{RF}$ should be given by  \ref{ENERGY} :
	
	\begin{equation*} \label{EP}
		E (g_{RF}) = 1 + (\beta -1) - \gamma = 1- \beta .
	\end{equation*}
	We expect that if we re-scale $g_{\epsilon}$ around small balls centered at any of the points $p_1, \ldots p_8$ we  get the metric $g_{RF}$ in the limit as $\epsilon \to 0$. We know what is the energy of the metrics $g_{\epsilon}$ and $g_0$:
	
	$$ E(g_{\epsilon}) = 3 + (\beta -1) \chi (C_{\epsilon})= 3 + (\beta -1) (-4) = 7 -4\beta  $$
	
	$$ E(g_0) = 3 + (\gamma -1) \chi (C_0)= 3 + (2\beta-2 ) 2 = 4\beta -1 . $$
	Our speculation is in agreement with the fact that
	
	\begin{equation*} \label{energy lost}
		E(g_{\epsilon}) - E(g_0) = 8 (1-\beta) =8 E(g_{RF}) .
	\end{equation*}
	When $Z \cap C_0$ consists of less than eight points, we might expect to see a bubble tree phenomena at the multiple points of the intersection.

\bibliographystyle{plain}
\bibliography{ref}

\end{document}